\documentclass[12pt]{amsart}
\usepackage{amsmath}
\usepackage{amssymb}
\usepackage{epsfig}
\usepackage{graphicx}
\usepackage{xcolor}
\usepackage{hyperref}
\numberwithin{equation}{section}

\input xy
\xyoption{all}

\calclayout
\allowdisplaybreaks[3]

\DeclareFontFamily{U}{mathb}{\hyphenchar\font45}
\DeclareFontShape{U}{mathb}{m}{n}{
      <5> <6> <7> <8> <9> <10> gen * mathb
      <10.95> mathb10 <12> <14.4> <17.28> <20.74> <24.88> mathb12
      }{}
\DeclareSymbolFont{mathb}{U}{mathb}{m}{n}
\DeclareMathSymbol{\righttoleftarrow}{3}{mathb}{"FD}

\newcommand{\eqto}{\stackrel{\lower1.5pt\hbox{$\scriptstyle\sim\,$}}\to}
\newcommand{\eqdashto}{\stackrel{\lower1.5pt\hbox{$\scriptstyle\sim\,$}}\dashrightarrow}

\newcommand{\actsfromright}{\righttoleftarrow}

\theoremstyle{plain}
\newtheorem{prop}{Proposition}[section]

\newtheorem{theo}[prop]{Theorem}
\newtheorem{coro}[prop]{Corollary}

\theoremstyle{definition}
\newtheorem{defi}[prop]{Definition}

\newtheorem{rema}[prop]{Remark}

\newtheorem{exam}[prop]{Example}

\makeatother
\makeatletter

\newcommand{\bC}{\mathbb C}
\newcommand{\bF}{\mathbb F}

\newcommand{\bN}{\mathbb N}
\newcommand{\bP}{\mathbb P}
\newcommand{\bQ}{\mathbb Q}
\newcommand{\bR}{\mathbb R}
\newcommand{\bZ}{\mathbb Z}

\newcommand{\cE}{\mathcal E}

\newcommand{\cI}{\mathcal I}
\newcommand{\cK}{\mathcal K}
\newcommand{\cO}{\mathcal O}

\newcommand{\rA}{\mathrm A}
\newcommand{\rD}{\mathrm D}
\newcommand{\rE}{\mathrm E}
\newcommand{\rH}{\mathrm H}
\newcommand{\rL}{\mathrm L}
\newcommand{\rN}{\mathrm N}
\newcommand{\rM}{\mathrm M}
\newcommand{\rU}{\mathrm U}

\newcommand{\Bl}{\operatorname{Bl}}

\newcommand{\rk}{\operatorname{rk}}

\newcommand{\dual}{\operatorname{dual}}

\newcommand{\Gr}{\operatorname{Gr}}
\newcommand{\Pic}{\operatorname{Pic}}

\newcommand{\Aut}{\operatorname{Aut}}

\newcommand{\NS}{\mathrm{NS}}
\newcommand{\Ex}{\mathrm{Ex}}
\newcommand{\Bir}{\mathrm{Bir}}

\newcommand{\ra}{\rightarrow}

\author{Brendan Hassett}
\address{Department of Mathematics\\
Brown University
Box 1917 \\
151 Thayer Street,
Providence, RI 02912 \\
USA}
\email{brendan\underline{ }hassett@brown.edu}

\author{Yuri Tschinkel}
\address{Courant Institute\\
                New York University \\
                New York, NY 10012 \\
                USA }
\email{tschinkel@cims.nyu.edu}

\address{Simons Foundation\\
160 Fifth Avenue\\
New York, NY 10010\\
USA}

\title{Involutions on K3 surfaces and derived equivalence}

\begin{document}
\date{\today}

\begin{abstract}
We study involutions on K3 surfaces under conjugation by derived 
equivalence and more general relations, together with applications
to equivariant birational geometry.  
\end{abstract}

\maketitle

\section{Introduction}
\label{sect:intro}

The structure of $\Aut D^b(X)$, 
the group of 
autoequivalences of the bounded derived category $D^b(X)$ of a K3 surface $X$, is very rich but well-understood only when the Picard group $\Pic(X)$ has rank one \cite{BayerBridgeland}. 
The automorphism group $\Aut(X)$ of $X$ lifts to 
$\Aut D^b(X)$, and one may consider the problem of classification of finite subgroups 
$G\subset \Aut(X)$ up to conjugation -- either by automorphisms, derived equivalence, or even larger groups.  
This problem is already interesting for cyclic $G$, and even for involutions, e.g., Enriques or Nikulin involutions. There is an extensive literature classifying 
these involutions on a given K3 surface $X$: topological types,
moduli spaces of polarized K3 surfaces with involution, and the involutions on a single $X$ up to automorphisms, see, e.g., \cite{AN}, \cite{Geemen-S}, \cite{Ohashi}, \cite{ShiV}, \cite{Zhang}.

Here we investigate involutions up to derived equivalence, i.e., derived equivalences respecting involutions. Our interest in ``derived" phenomena was sparked by a result in \cite{Sosna}---there exist complex conjugate, derived equivalent nonisomorphic K3 surfaces---as well as our investigations of arithmetic problems on K3 surfaces \cite{HT-derived}, \cite{HT22}.

One large class of involutions $\sigma:X\rightarrow X$ are those whose quotient $Q=X/\sigma$ is rational. Examples include
$Q$ a del Pezzo surface and $X \rightarrow Q$ a double cover 
branched along a smooth curve $B \in |-2K_Q|$. We may allow
$Q$ to have ADE surface singularities away from $B$, or $B$ to have
ADE curve singularities; then we take $X$ as the minimal 
resolution of the resulting double cover of $Q$. These
were studied by Alexeev and Nikulin in connection with classification questions concerning singular del Pezzo surfaces \cite{AN}. 
Our principal result here (see Section~\ref{sect:geni}) is that
\begin{itemize}
\item equivariant derived
equivalences of such $(X,\sigma)$ are in fact equivariant isomorphisms  
(see Corollary~\ref{coro:antisyminv}).  \end{itemize}

Our study of stable equivalence of lattices with involution
leads us to a notion of {\em skew equivalence},
presented in Section~\ref{sect:skew}. Here, duality
interacts with the involution which is reflected 
in a functional equations for the Fourier-Mukai kernel.
Explicit examples, for anti-symplectic actions with
quotients equal to $\bP^2$, are presented in Section~\ref{sect:ratq}.

Next, we focus on {\em Nikulin} involutions $\iota : X\to X$, 
i.e., involutions preserving the symplectic form, so that the resolution of singularities $Y$ of the 
resulting quotient $X/\iota$ is a K3 surface. A detailed study of such involutions can be found in \cite{Geemen-S}. In addition to the polarization class, the Picard group  $\Pic(X)$ contains the lattice $\rE_8(-2)$; van Geemen and Sarti describe the moduli and the geometry in the case of minimal Picard rank $\rk \Pic(X)=9$. In Section~\ref{sect:nik}, we extend their results to higher ranks, and 
\begin{itemize}
\item 
exhibit nontrivial derived equivalences between Nikulin involutions (Proposition~\ref{prop:nik-der}). 
\end{itemize}
These, in turn, allow us to construct in Section~\ref{sect:geom} examples of equivariant birational isomorphisms $\phi:\bP^4\dashrightarrow \bP^4$ 
with nonvanishing invariant $C_G(\phi)$, introduced in \cite{LSZ}, \cite{LSh} and extended to the equivariant context in \cite{KT-map}.

The case of {\em Enriques} involutions $\epsilon:X\to X$, i.e., fixed-point free involutions, so that the resulting quotient $X/\epsilon$ is an Enriques surface, has also received considerable attention.
There is a parametrization of such involutions
in terms of the Mukai lattice $\widetilde{\rH}(X)$, 
and an explicit description of conjugacy classes, up to automorphisms $\Aut(X)$, in interesting special cases, e.g., for K3 surfaces of Picard rank 11, 
Kummer surfaces of product type, general Kummer surfaces, or singular K3 surfaces
\cite{kondo}, \cite{Ohashi}, \cite{Sert}, \cite{ShiV}. 
In Section~\ref{sect:genEnr} we observe that 
\begin{itemize}
\item the existence of an Enriques involution on a K3 surface $X$ implies that every derived equivalent surface is equivariantly isomorphic to $X$ (Propositions~\ref{prop:enr-derived} and \ref{prop:any});
\item while there are {\em no} nontrivial equivariant derived autoequivalences, we exhibit nontrivial {\em orientation reversing} (i.e., skew) equivalences, e.g., on singular K3 surfaces.
\end{itemize}

\

\noindent
{\bf Acknowledgments:} 
The first author was partially supported by Simons Foundation Award 546235 and NSF grants 1701659 and 1929284,
the second author by NSF grant 2000099.  We are grateful to Nicolas Addington, Arend Bayer, Igor Dolgachev, Sarah Frei, Lisa Marquand, and Barry Mazur for helpful suggestions.

\section{Lattice results}
\label{sect:lattices}

We recall basic terminology and results concerning lattices: torsion-free finite-rank abelian groups $\mathrm L$ together with a nondegenerate integral quadratic form $(\cdot, \cdot)$, which we assume to be even. Basic examples are 
$$
\mathrm U =  \left( \begin{matrix} 0 & 1 \\ 1 & 0 \end{matrix} \right)
$$
and positive definite lattices associated with Dynkin diagrams (denoted by the same letter). 

We write $\mathrm L(2)$, when the form is multiplied by $2$. We let 
$$
d(\rL):=\rL^*/\rL
$$ 
be the discriminant group and 
$$
q_{\rL}:d(\rL)\to \bQ/2\bZ
$$ 
the induced discriminant quadratic form.

\subsection*{Nikulin's form of Witt cancellation:}

\begin{prop} \cite[Cor.~1.13.4]{nik-lattice} \label{prop:nikulin1}
Given an even lattice $\mathrm L$, $\mathrm L \oplus \mathrm U$ is the unique lattice with
its signature and discriminant quadratic form.
\end{prop}

If lattices $\mathrm L_1$ and $\mathrm L_2$ are stably isomorphic -- become isomorphic
after adding unimodular lattices of the same signature -- then 
$$
\mathrm L_1 \oplus \mathrm U \simeq \mathrm L_2 \oplus \mathrm U.$$

\subsection*{Nikulin stabilization result:}
Given a lattice $\rL$, write $\rL\otimes \bZ_p$
for the induced $p$-adic quadratic form. The genus of $\rL$
is the collection of all lattices equivalent to $\rL$ over
$\bZ_p$ for each prime $p$ and over $\bR$. Stably equivalent
lattices are in the same genus. The $p$-primary
part of $d(\rL)$ depends only on $\rL\otimes \bZ_p$ and is written
$d(\rL\otimes \bZ_p)$. We use $q_{\rL\otimes \bZ_p}$ for the induced discriminant quadratic
form on $d(\rL\otimes \bZ_p)$, with values in the $p$-primary part of
$\bQ/\bZ$ for odd $p$; when $\rL$ is even and $p=2$ it takes values
in the $2$-primary part of $\bQ/2\bZ$. 
For a finitely generated abelian group $A$,
let $\ell(A)$ denote the minimal number of generators.

\begin{prop} \cite[Thm.~1.14.2]{nik-lattice} \label{prop:nikulin2}
Let $\rL$ be an even indefinite lattice satisfying
\begin{itemize}
\item{$\operatorname{rank}(\rL) \ge \ell(d(\rL\otimes \bZ_p))+2$ for all $p\neq 2$;}
\item{if $\operatorname{rank}(\rL) = \ell(d(\rL\otimes \bZ_2))$ then $q_{\rL\otimes \bZ_2}$
contains $u^{(2)}_+(2)$ or $v^{(2)}_+(2)$ as a summand, i.e., the
discriminant quadratic forms of
$$
\rU^{(2)}(2)=\left( \begin{matrix} 0 & 2 \\ 2 & 0 \end{matrix} \right), \quad
\mathrm V^{(2)}(2) = \left( \begin{matrix} 4 & 2 \\ 2 & 4 \end{matrix} \right).
$$
}
\end{itemize}
Then the genus of $\rL$ admits a unique class and 
$\mathrm O(\rL) \rightarrow \mathrm O(q_{\rL})$ is surjective.
\end{prop}
\begin{rema} \cite[Rem.~1.14.5]{nik-lattice} \label{rema:niksmall}
The $2$-adic condition can be achieved whenever the discriminant
group $d(\rL)$ has $(\bZ/2\bZ)^3$ as a summand.
\end{rema}

Thus given a lattice $\mathrm L$, any automorphism of $(d(\mathrm L),q_{\mathrm L})$ may be achieved
via an automorphism of $\mathrm L\oplus \mathrm U$. More precisely, given two lattices
$\mathrm L_1$ and $\mathrm L_2$ of the same rank and signature and an isomorphism
$$\varrho: (d(\mathrm L_1),q_{\mathrm L_1}) \stackrel{\sim}{\longrightarrow} (d(\mathrm L_2),q_{\mathrm L_2})$$
there exists an isomorphism
$$\rho: \mathrm L_1\oplus \rU \stackrel{\sim}{\longrightarrow} 
\mathrm L_2 \oplus \rU
$$
inducing $\varrho$.

\subsection*{Nikulin imbedding result:}

\begin{prop} \cite[Cor.~1.12.3,Thm.~1.14.4]{nik-lattice} \label{prop:nikulin3}
Let $\rL$ be an even lattice of signature $(t_+,t_-)$ and discriminant
group $d(\rL)$. Then $\rL$ admits a primitive embedding into a unimodular even lattice of signature $(\ell_+,\ell_-)$ if
\begin{itemize}
\item{$\ell_+ - \ell_- \equiv 0 \mod 8$;}
\item{$\ell_+\ge t_+$ and $\ell_- \ge t_-$;}
\item{$\ell_+ + \ell_- - t_+ - t_- > \ell(d(\rL))$,
the rank of $d(\rL)$.}
\end{itemize}
This embedding is unique up to automorphisms if
\begin{itemize}
\item{$\ell_+> t_+$ and $\ell_- > t_-$;}
\item{$\ell_+ + \ell_- - t_+ - t_- \ge 2+ \ell(d(\rL))$.}
\end{itemize}
\end{prop}
In particular, {\em any} even nondegenerate lattice of signature $(1,9)$ admits a unique embedding into the 
K3 lattice $\rU^{\oplus 3} \oplus \rE_8(-1)^{\oplus 2}.$

\section{Mukai lattices and derived automorphisms}
\label{sect:genK3}

Throughout, we work over the complex numbers $\bC$. Let $X$ be a projective K3 surface and 
$$
\Pic(X)\subset \rH^2(X,\bZ) \simeq \mathrm{E}_8(-1)^{\oplus 2} \oplus \mathrm U^3
$$
its Picard lattice, a sublattice of a lattice of signature $(3,19)$, with respect to the intersection pairing. The Picard lattice governs the automorphisms of $X$. The composition
$$\varpi: \Aut(X) \rightarrow \mathrm{O}(\rH^2(X,\bZ))\rightarrow \mathrm{O}(\Pic(X))$$
has finite cyclic kernel \cite{Nik-finite,kondo}. 
The image can be computed explicitly, at least up to finite subgroups,
in terms of $\Pic(X)$ \cite[\S 2]{LP}. Consider the subgroup generated
by reflections in $(-2)$-classes, i.e., indecomposable effective divisors of self-intersection $-2$; it acts naturally on the positive cone in $\Pic(X)_{\bR}$. Then
the image of $\varpi$ is a finite-index subgroup of those elements 
leaving invariant a fundamental domain for this action, i.e. the ample cone. 
All possible finite $G\subset \Aut(X)$ have been classified, see \cite{BH}. Classification of $\Aut(X)$-conjugacy classes of elements or subgroups boils down to lattice theory of $\Pic(X)$;  we will revisit it in special cases below.

The {\em transcendental lattice} of $X$ is the orthogonal complement
$$
T(X):=\Pic(X)^\perp \subset \rH^2(X,\bZ). 
$$
This lattice plays a special role: 
two K3 surfaces $X_1$, $X_2$ are {\em derived equivalent} if and only if there exists an isomorphism of lattices
$$
T(X_1)\stackrel{\sim}{\longrightarrow} T(X_2),
$$
compatible with Hodge structures \cite{Orlov}. 
Derived equivalence also means that
the lattices $\Pic(X_1)$ and $\Pic(X_2)$ belong to the same genus. 
Over nonclosed fields, or in equivariant contexts, derived equivalence is a subtle property, see, e.g., \cite{HT-derived}, \cite{HT22}. 

We recall standard examples of Picard lattices of derived equivalent but not isomorphic K3 surfaces 

\begin{rema} \label{rema:FMpartners}
In Picard rank one: the number of nonisomorphic derived equivalent surfaces is governed by the number of prime divisors of the polarization degree $2d$; see \cite[Cor.~2.7]{HLOY} and 
Proposition~\ref{prop:nikulin2}.
The isomorphisms classes correspond
to solutions of the congruence
\begin{equation} \label{squarecond}
x^2 \equiv 1 \pmod{4d} 
\end{equation}
modulo $\pm 1$.
When $d>1$ the number of derived equivalent K3 surfaces
is $2^{\tau(d)-1}$, where $\tau$ is the number of distinct 
prime factors of $d$.  

In Picard rank two: derived equivalences among lattice-polarized K3 surfaces of square-free discriminant are governed by the genera in
the class group of the corresponding real quadratic field
\cite[Sect.~3]{HLOY}. 
\end{rema}
Here are instances where derived equivalence is trivial
\begin{prop} \cite[Cor.~2.6, 2.7]{HLOY}
\label{prop:dereq}
Derived equivalence implies isomorphism in each of the following cases:
\begin{itemize}
\item{if the Picard rank is $\ge 12$;}
\item{if the surface admits an elliptic fibration with a section;}
\item{if the Picard rank is $\ge 3$ and the discriminant group of
the Picard group is cyclic.}
\end{itemize}
\end{prop}
We give a further example in Proposition~\ref{prop:nik-der}.

Let 
$$
\widetilde{\rH}(X) :=\rH^0(X,\bZ)(-1)\oplus \rH^2(X,\bZ)\oplus \rH^4(X,\bZ)(1)
$$
be its {\em Mukai lattice}, a lattice of signature $(4,20)$, with respect to the Mukai pairing. There is a surjective homomorphism 
\cite[Cor.~3]{HMS}
$$
\Aut D^b(X) \to \mathrm O^+(\widetilde{\rH}(X)) \subset 
\mathrm O(\widetilde{\rH}(X))
$$
onto the group of {\em signed} Hodge isometries, a subgroup of 
the orthogonal group of the Mukai lattice preserving orientations
on the positive $4$-planes.

We retain the notation from \cite[Cor.~3]{HT22}, where we discussed the notion and basic properties of equivariant derived equivalences between K3 surfaces. We recall: 

\

\begin{quote}
Let $X_1$ and $X_2$ be K3 surfaces equipped with a generically free action of a finite cyclic group $G$. Then $X_1$ and $X_2$ are $G$-equivariantly derived equivalent if and only if there exists a $G$-equivariant isomorphism of their Mukai lattices 
$$
\widetilde\rH(X_1)  \stackrel{\sim}{\longrightarrow} \widetilde{\rH}(X_2)
$$
respecting the Hodge structures. 
\end{quote}
Note that the $G$-action is necessarily trivial on 
$$
\rH^0(X,\bZ)(-1)\oplus \rH^4(X,\bZ)(1).
$$

Even in the event of an isomorphism $X_1\simeq X_2$, 
equivariant derived equivalences are interesting: indeed, there are actions of finite groups $G$ that are not conjugate in $\Aut(X)$ but are conjugate via $\Aut D^b(X)$ as the action of the latter group is visibly larger.  See Proposition~\ref{prop:nik-derived-ten} for examples.

Let $G$ be a finite group and $X_1$ and $X_2$ K3 surfaces with $G$-actions.
For simplicity, assume that $G$ acts on $T(X_i)$ via $\pm \mathrm I$. (This is the case if the transcendental cohomology is simple.)  Given a $G$-equivariant isomorphism $T(X_1) \simeq T(X_2)$, can we lift to a $G$-equivariant isomorphism of Mukai lattices
$$
\widetilde{\rH}(X_1,\bZ) \simeq \widetilde{\rH}(X_2,\bZ),
$$
where $G$ acts trivially on the hyperbolic summand
$$
	\rU = \rH^0 \oplus \rH^4?
$$

Clearly the answer is NO. Suppose that $G=C_2=\left< \epsilon \right>$ and the $\epsilon=-1$ eigenspaces are stably isomorphic but not isomorphic. Adding $\rU$ does nothing to achieve the desired stabilization. 
In other words, $\rU$ is ``too small''. We need to add summands where $G$ acts {\em nontrivially} to achieve stabilization across all the various isotypic components.
See Proposition~\ref{prop:bootstrap} for more on this
question.

\section{Generalities concerning involutions on K3 surfaces}
\label{sect:geni}

Let $i:X \ra X$ be an involution on a complex projective
K3 surface, which acts faithfully on $\rH^2(X,\bZ)$ by the Torelli
Theorem. It is {\em symplectic} (resp.~{\em anti-symplectic}) if
$$
i^*\omega = \omega  \quad (\text{resp. } -\omega),
$$
where $\omega$ is a holomorphic two-form.  
Nikulin \cite{Nik-finite} showed that any symplectic involution
fixes eight isolated points and that all such involutions
are topologically conjugate; these are the {\em Nikulin involutions}
studied in Section~\ref{sect:nik}. An involution without
fixed points was classically known to be an {\em Enriques involution}
arising from a double cover $X \ra S$ of an Enriques surface.

The case of anti-symplectic involutions with fixed points is more complicated. Nikulin enumerated 74 cases beyond the Enriques case; see \cite{AN,BH,AE,alexeev} for details of the various cases.

Given an anti-symplectic involution $i:X\rightarrow X$ on a K3 surface, we recall the Nikulin invariants  $(r,a,\delta)$ \cite[\S 2]{AE}:
Let $r$ denote the rank of the lattice 
$$
S=\rH^2(X,\bZ)^{i=1},
$$ 
which is indefinite if $r>1$. We are using the fact that
transcendental classes of $X$ are anti-invariant under $i$, as
the quotient $X/i$ admits no holomorphic two-form.  
We write
$$T=\rH^2(X,\bZ)^{i=-1}=S^{\perp}$$ 
for the complementary lattice with signature $(2,20-r)$, 
which is indefinite if $r<20$.
The discriminant group $d(S) \simeq d(T)$ is a $2$-elementary group;
its rank is denoted by $a$. This group comes with a quadratic form $$q_S:d(S) \rightarrow \bQ/2\bZ.$$  
The {\em coparity} $\delta$ equals $0$ if $q_S(x) \in \bZ$ for each 
$x \in d(S)$ and equals $1$ otherwise.  

We relate this to geometric invariants. For an anti-symplectic involution, there are no isolated fixed points so the fixed locus $R=X^i$ is of pure dimension one or empty. Suppose
there are $k+1$ irreducible components, with genera summing to $g$. 
Then we have cf. \cite[p.5]{AE}
$$g = 11 - (r+a)/2 \quad k = (r-a)/2,$$
excluding the Enriques case $(r,k,\delta) = (10,10,0)$.  

Nikulin classifies even indefinite $2$-elementary lattices $\rL$. They
are determined uniquely by $(r,a,\delta)$ and $\mathrm O(\rL) \rightarrow \Aut(d(\rL))$ is surjective.  
In the definite case, {\em a priori} there are multiple classes in 
each genus but this is not relevant for our applications. Indeed, the possibilities include
\begin{itemize}
\item
$r=a=1$: $X$ is a double cover of $\bP^2$ branched along a sextic
plane curve.
\item
The case where $T$ is definite $(r=20, a=2, g=0, k=9)$, we have
$d(T)=\bZ/2\bZ \oplus \bZ/2\bZ$ thus is equal to 
$$
	\left( \begin{matrix} 2 & 0 \\ 0 &  2 \end{matrix} \right).
$$
Even in this case, automorphisms of the discriminant group are realized by automorphisms of the lattice. 
\end{itemize}
\begin{theo}[Alexeev-Nikulin] For each admissible set of invariants $(r,a,\delta)$, there is a unique
orthogonal pair of lattices (S,T) embedded in the K3 lattice $\Lambda$, up to automorphisms of $\Lambda$. 
There are 75 such cases.
\end{theo}

\begin{coro} 
\label{coro:antisyminv}
Any equivariant derived equivalence of K3 surfaces with anti-symplectic involutions induces an equivariant isomorphism between the underlying K3 surfaces.  
\end{coro}
\begin{proof}
Suppose that $(X_1,i_1)$ and $(X_2,i_2)$ are derived equivalent,
compatibly with their anti-symplectic involutions.

Indeed, derived equivalence shows that the invariant 
(resp. anti-invariant)
sublattices of the Picard group are stably equivalent 
(resp. equivalent): 
$$\Pic(X_1)^{i_1=1} \oplus \rU \simeq \Pic(X_2)^{i_2=1}\oplus \rU, \quad
\Pic(X_1)^{i_1=-1} \simeq \Pic(X_2)^{i_2=-1}.$$
Since the
possibilities for the invariant sublattices
are characterized by their $2$-adic invariants, we have
$$\Pic(X_1)^{i_1=1} \simeq \Pic(X_2)^{i_2=1}.$$
We have already observed that all the possible isomorphisms between their discriminants
$$\left(d(\Pic(X_1)^{i_1=1}),q_1\right) \simeq
\left(d(\Pic(X_2)^{i_2=1}),q_2\right)
$$
are realized by isomorphisms of the lattices. In particular,
there exists a choice compatible with the isomorphism
$$\rH^2(X_1,\bZ)^{i_1=-1} \stackrel{\sim}{\ra} \rH^2(X_2,\bZ)^{i_2=-1}$$
induced by the derived equivalence. Thus we obtain
isomorphisms on middle cohomology, compatible with the involutions. 
The Torelli Theorem gives an isomorphism $X_1 \stackrel{\sim}{\ra} X_2$ 
respecting the involutions.
\end{proof}

\begin{coro}
Let $(X_1,\sigma_1)$ and $(X_2,\sigma_2)$ denote K3 surfaces 
with involutions that are $C_2$-equivariantly derived equivalent.
If $X_1/\sigma_1$ is rational then $X_2/\sigma_2$ is rational as well. 
\end{coro}

Indeed, the rationality of the quotient forces the involution to be
anti-symplectic.  

\begin{exam}
Having an anti-symmetric involution 
is {\em not} generally a derived property. 
For example, consider Picard lattices
$$A_1 = \left( \begin{matrix} 2 & 13 \\ 13 & 12 \end{matrix} \right)
\quad
A_2 = \left(\begin{matrix} 8 & 15 \\ 15 & 10 \end{matrix} \right).
$$
These forms are stably equivalent but not isomorphic.
As in Remark~\ref{rema:FMpartners} -- see \cite[Sec.~2.3]{HT-derived}
for details -- choose derived equivalent 
K3 surfaces
$X_1$ and $X_2$ with $\Pic(X_1)=A_1$ and $\Pic(X_2)=A_2$.  
Note that $A_2$ does not represent two and admits no involution
acting via $\pm 1$ on $d(A_2)$; thus $X_2$ does not admit
an involution.  
\end{exam}
This should be compared with Proposition~\ref{prop:enr-derived}:
Having an Enriques involution is a derived invariant.

\

We collect some lattice-theoretic observations that
will serve as a foundation for Section~\ref{sect:skew}:

\begin{prop} \label{prop:bootstrap}
Let $(X_1,i_1)$ and $(X_2,i_2)$ be K3 surfaces
with involutions, both symplectic or anti-symplectic. Extend the involutions to actions
on the Mukai lattices 
$$
\widetilde{i_j}: \widetilde{\rH}(X_j,\bZ), \quad j=1,2,
$$
where 
$$\widetilde{i_j}|\rH^{k} = \begin{cases}
i_j^* & \text{ if } k=2, \\
-\mathrm{I}  & \text{ if } k=0,4.
\end{cases}
$$
An equivalence of such actions, on Hodge structures
of weight two, corresponds to a triple
\begin{enumerate}
\item{an isomorphism of Hodge structures $$
t:T(X_1)\rightarrow T(X_2),
$$
}
\item{an isomorphism of lattices
$$\pi^{+1}:\Pic(X_1)^{i_1=1} \rightarrow \Pic(X_2)^{i_2=1},
$$}
\item{a stable equivalence of lattices 
$$\pi^{-1}:\Pic(X_1)^{i_1=-1} \oplus \rU \rightarrow
\Pic(X_2)^{i_2=-1} \oplus \rU,
$$}
\end{enumerate}
satisfying the following conditions
\begin{itemize}
\item{the isomorphisms induced by $\pi^{\pm 1}$
on discriminant groups agree on the images
$$\Pic(X_j) \rightarrow 
d(\Pic(X_j)^{i_j=1})\oplus d(\Pic(X_j)^{i_j=-1}),$$
which are $2$-elementary groups;}
\item{the resulting isomorphism
$$\Pic(X_1) \rightarrow \Pic(X_2)$$
is compatible with $t$ on discriminant groups.
}
\end{itemize}
\end{prop}
This is proven through two applications of Nikulin's lattice extension theory, first to the Picard group and then to the full cohomology lattice.   

Fixing $T(X_j)$ and $\Pic(X_j)^{i_j=1}$,
the possible equivalences are indexed by isomorphisms
$\pi^{-1}$ restricting to the identity on the
distinguished $2$-elementary subgroups. 
Applying Nikulin stabilization, the
equivalent Mukai lattices, with these data,
are indexed by the stable isomorphism classes
of the anti-invariant Picard groups, where the
stable isomorphism restricts to the identity
on the distinguished $2$-elementary subgroup of the
their discriminant groups.  
\begin{coro} \label{coro:counting}
Suppose the anti-invariant Picard lattice $P$
is unique in its genus. Then the possible Mukai lattices
$(\widetilde{\rH},\widetilde{i})$ with $P$ 
are indexed by automorphisms of $d(P)$ restricting
to the identity on the two-elementary subgroup.
\end{coro}
This is an equivariant version of the counting
results of \cite{HLOY}.

\section{Cohomological Fourier-Mukai transforms}
\label{sect:cfm}

Let $X_1$ and $X_2$ be smooth projective complex K3 surfaces. A fundamental
result of Orlov \cite{Orlov} shows that any equivalence 
$$\Phi: D^b(X_1) \rightarrow D^b(X_2)$$
arises from a kernel $\cK \in D^b(X_1 \times X_2)$ through a Fourier-Mukai
transform
$$
\begin{array}{rcl}
\Phi_{\cK}: D^b(X_1) & \ra & D^b(X_2) \\
\cE & \mapsto & {\pi_2}_* (\pi_1^* \cE \otimes \cK).
\end{array}
$$
All the indicated functors are taken in their derived senses. 
Given such a kernel, there is also a Fourier-Mukai transform in the 
opposite direction
$$
\begin{array}{rcl}
\Psi_{\cK}: D^b(X_2) & \ra & D^b(X_1) \\
\cE & \mapsto & {\pi_1}_* (\pi_2^* \cE \otimes \cK).
\end{array}
$$

Mukai has computed the kernel of the inverse
$$
\Phi_{\cK}^{-1}=\Psi_{\cK^{\vee}[2]}
$$
i.e., a shift of the dual to our original kernel. 
See \cite[4.10]{MukaiTata}, \cite[\S~4.3]{BBHR}, and \cite[p.~133]{HuFMTAG}
for details. The computation relies on Grothendieck-Serre Duality,
so the appearance of the dualizing complex is natural.
This machinery \cite[\S~3.4]{HuFMTAG} also allows us to analyze how Fourier-Mukai transforms interact with taking duals:
\begin{align*}
\Phi_{\cK} (\cE^{\vee}) &= {\pi_2}_*(\cK \otimes \pi_1^*(\cE^{\vee})) \\
		&= (( {\pi_2}_*(\cK^{\vee} \otimes \pi_1^*\cE) )^{\vee})[-2]  \\
		&=  ((\Phi_{\cK^{\vee}}\cE)[2])^{\vee} \\
		&=  (\Phi_{\cK^{\vee}[2] } \cE)^{\vee}
\end{align*}

Suppose that $X_1$ and $X_2$ are equivalent through an isomorphism
$$
X_2=M_H(X_1,v_1),
$$
i.e., the moduli space of sheaves $\cE_p,p\in X_2$, on $X_1$ with Mukai vector
$$
v_1=v(\cE_p)=(r,D,s) \in \widetilde{\rH}(X,\bZ),
$$
Gieseker-stable with respect to some polarization $H$
on $X_1$. 
Here $r$ is the rank of $\cE_p$, $D=c_1(\cE_p)$, and $s=\chi(\cE_p)-r$.
We assume there exists another Hodge class $v' \in \widetilde{\rH}(X_1,\bZ)$
such that $\left<v,v'\right>=1$; in particular, $v$ is primitive. (For information on how to realize
derived equivalences via such moduli spaces, see
\cite[\S 4,5]{MukaiTata} and \cite[p.~385]{HuyJAG}, the discussion
following Proposition 4.1.)
Let $\cE \rightarrow X_1\times X_2$ denote a universal sheaf;
by simplicity of the sheaves, $\cE$ is unique up to 
tensoring by a line bundle from $X_2$. We may use $\cE$ as a kernel
inducing a derived equivalence between $X_1$ and $X_2$ 
\cite[10.25]{HuFMTAG}.  
Our formulas for inverses are compatible 
with tensoring the kernel by line bundles from one of the factors.

In searching for Fourier-Mukai kernels, cohomological Fourier-Mukai
transforms play a crucial role.  Let $\omega_i \in \rH^4(X_i,\bZ)$
denote the point class and set \cite[{\S}1]{MukaiTata}, \cite[p.~128]{HuFMTAG}
$$Z_{\cK} := \pi_1^*(1+\omega_1) \operatorname{ch}(\cK) \pi_2^*(1+\omega_2) \in \rH^*(X_1\times X_2, \bZ),$$
where the middle term is the Chern character.
Then $Z_{\cK}$ induces an integral isomorphism of Hodge structures
$$
\phi_{\cK}: \widetilde{\rH}(X_1,\bZ) \stackrel{\sim}{\longrightarrow}
\widetilde{\rH}(X_2,\bZ)
$$
compatible with Mukai pairings; this is called the {\em cohomological
Fourier-Mukai transform}. For $\cE \in D^b(X_1)$, we have the identity
$$
\phi_{\cK}(v(\cE))=v(\Phi_{\cK}(\cE)).
$$
We use $\psi_{\cK}$ to denote the cohomological transform of $\Psi_{\cK}$. 

Most cohomological Fourier-Mukai transforms are
induced by kernels
\begin{prop}  \cite{Orlov,HMS}
\label{prop:getkernel}
Given an orientation-preserving integral Hodge isometry
$$\phi:\widetilde{\rH}(X_1,\bZ) \rightarrow
\widetilde{\rH}(X_2,\bZ)$$
there exists a derived equivalence
$$\Phi_{\cK}:D^b(X_1) \rightarrow D^b(X_2)$$
such that $\phi$ is the cohomological Fourier-Mukai transform of $\Phi_{\cK}$.  
\end{prop}

Suppose that $(X_1,f_1)$ is a polarized K3 surface of degree $2r_0s$,
where $r_0$ and $s$ are relatively prime positive integers. 
Let $d_0$ be an integer prime to $r_0$ and fix the isotropic Mukai vector
$$v_0=(r_0,d_0f_1,d_0^2s) \in \widetilde{\rH}(X_1,\bZ).$$
Since $r_0$ and $d_0^2s$ are relatively prime, there exists
a Mukai vector $v'=(m,0,n)$ such that $\left<v_0,v'\right>=1$. 
Let $X_2=M_{f_1}(X_1,v_0)$ be the moduli space of torsion-free sheaves with Mukai vector $v_0$, Gieseker-stable with respect to $f_1$ --  also a K3 surface. Choose a universal sheaf
$\cE \rightarrow X_1 \times X_2$. Our goal is to describe
the induced isomorphism 
$$
\phi_{\cE}: \widetilde{\rH}(X_1,\bZ) \stackrel{\sim}{\rightarrow} 
\widetilde{\rH}(X_2,\bZ).
$$
Following \cite[Ch.~8]{HLbook} and \cite[{\S}2]{Yosh}, the polarization on $X_2$ is given by 
$$\det({\pi_2}_*(\cE \otimes \cO_H(s(r_0-2d_0))))^{\vee}, \quad H \in |f_1|,$$
a primitive ample divisor $f_2$ on $X_2$. More generally, we have
an isomorphism of Hodge structures
$$\rH^2(X_2,\bZ) = (v_0^{\vee})^{\perp} / \bZ v_0^{\vee}, $$
where the perpendicular subspace is taken with respect to the
Mukai pairing.  

\begin{prop} \cite{Yosh}
\label{prop:yoshioka}
Let $(X_1,f_1)$ and $(X_2,f_2)$ be K3 surfaces of Picard
rank one with $X_2\simeq M_{f_1}(X_1,v_0)$ as above.
Choose integers $d_1$ and $\ell$ such that $sd_0d_1-r_0\ell=1$
and take $\cK=\cE\otimes \pi_2^*L$ for some 
line bundle $L$ on $X_2$.
With respect to the bases
$$ (1,0,0), (0,f_j,0), (0,0,1) \in \widetilde{\rH}(X_j,\bZ)\cap \widetilde{\rH}^{1,1}(X_j)$$
the matrix of the cohomological Fourier-Mukai transform takes the form
\begin{equation} \label{CFMmatrix}
\phi_{\cK}:=\left( \begin{matrix}
d_0^2 s & 2d_0sr_0 & r_0  \\
d_0\ell & 2d_0d_1s-1 & d_1 \\
\ell^2r_0 & 2d_1s\ell r_0 & d_1^2s 
\end{matrix}
\right).
\end{equation}
\end{prop}
The inverse is obtained reversing the sign of the middle basis vector and interchanging the
role of $d_0$ and $d_1$:
$$
\left( \begin{matrix}
d_0^2 s & 2d_0sr_0 & r_0  \\
d_0\ell & 2d_0d_1s-1 & d_1 \\
\ell^2r_0 & 2d_1s\ell r_0 & d_1^2 s 
\end{matrix}
\right)
\left( \begin{matrix}
d_1^2 s & -2d_1sr_0 & r_0  \\
-d_1\ell & 2d_0d_1s-1 & -d_0 \\
\ell^2r_0 & -2d_0s\ell r_0 & d_0^2 s
\end{matrix}
\right) = \mathrm I.$$
The formula
$$
\phi_{\cK} \psi_{\cK^{\vee}}=\mathrm I
$$
is the cohomological realization of the identity
$$\Phi_{\cK} \Psi_{\cK^{\vee}[2]}=\mathrm I.$$ 
The third column of $\phi_{\cK}^{-1}$ is the Mukai vector $v_0^{\vee}$,
as
$$
\Phi_{\cK}^{-1}(\cO_p) = \cE_p^{\vee}, \quad 
p=[\cE_p]\in X_2=M_{f_1}(X_1,v).$$ 

\begin{rema} \label{rema:fromrankone}
The assumption in 
Proposition~\ref{prop:yoshioka} on the rank of the Picard groups is not too restrictive, as 
Proposition~\ref{prop:getkernel} allows us to specialize from
the rank-one case. Derived equivalences satisfying (\ref{CFMmatrix}) exist provided the primitive cohomology
groups are isomorphic
$$\rH^2(X_1,\bZ) \supset f_1^{\perp} \simeq f_2^{\perp} \subset
\rH^2(X_2,\bZ)$$
as integral Hodge structures. However, these are not
given as kernels associated with explicit moduli spaces
of sheaves Gieseker-stable with respect to some polarization.
\end{rema}

\begin{exam}
\label{exam:12}
Suppose that $(X_1,f_1)$ is a degree $12$ K3 surface. Consider
the isotropic Mukai vector $v=(2,f_1,3)$ so that
$$X_2:=M_{f_1}(X_1,v)$$
is also a K3 surface derived equivalent to $X_1$.  
Taking
$$
r_0=2, \quad s=3,\quad  d_0=1,\quad d_1=\ell=1,
$$
we obtain
\begin{align*}
(1,0,0) &\mapsto (3,f_2,2) \\
(0,f_1,0) &\mapsto (12,5f_2,12)\\
(0,0,1) & \mapsto (2,f_2,3) 
\end{align*}
with matrix
\begin{equation}
\varphi := \left( \begin{matrix}  3 & 12 & 2\\
						 1 &   5 & 1\\
						 2 & 12 & 3 
			\end{matrix} \right). 
\label{phidegree12}
\end{equation}
The determinant is $1$ with one eigenvector $(1,0,-1)$ with eigenvalue $1$; thus this
is orientation preserving.  
Note that
$$(2,-f_1,3) \mapsto (0,0,1)$$
whence
$$X_1=M_{f_2}(X_2,(2,f_2,3)), \quad X_2=M_{f_2}(X_1,(2,-f_1,3)).$$
The fact that $(1,0,-1)$ has eigenvalue $1$ gives
$$X_1^{[2]} \stackrel{\sim}{\dashrightarrow} X_2^{[2]}.$$

\end{exam}

\section{Locally-free kernels and wall-crossing}
\label{sect:crossing}

For applications to skew equivalence, discussed in Section~\ref{sect:skew}, we require derived equivalences
between K3 surfaces $X_1$ and $X_2$ induced by
locally-free kernels 
$$\cE \rightarrow X_1 \times X_2.$$
Many equivalences do not arise in this way.

\begin{exam}
Suppose that $X_1=X_2=X$ and consider the equivalence
arising by interpreting $X$ as the moduli space of 
ideal sheaves $I_x, x\in X$.  These sheaves are not
locally-free.    
\end{exam}

We refer the reader to \cite[\S 1.2]{HLbook} for 
the definitions and background on $\mu$-stable
sheaves and relations to Gieseker stability.
We use the implications \cite[Lemma~1.2.13]{HL}
$$\mu_H\text{-stable} \Rightarrow \text{stable wrt $H$} \Rightarrow
\text{semistable wrt $H$} \Rightarrow \mu_H
\text{-semistable}.$$

Now $\mu$-stable sheaves on K3 surfaces are
typically locally-free: Let $E$ be a simple sheaf
on a K3 surface $X$, with $v(E)$ isotropic, 
such that $\mu_H$-stable for some
polarization $H$. Then $E$ is locally-free, with the
exception of ideal sheaves $I_x$ \cite[3.10]{MukaiTata},
\cite[6.1.9]{HL}.
The problem is that moduli spaces $M_H^{\mu s}(X,v)$
of such sheaves are not always compact, when there
are strictly $\mu_H$-semistable sheaves.

We recall criteria guaranteeing that $\mu$ stability
and semistability coincide. Let $v=(r,D,s)$ be a primitive 
isotropic Mukai vector of rank $r>0$ for a K3 surface
$X$. Assume that $D$ is primitive and $H$ is a polarization
avoiding ``walls'', i.e., hyperplanes expressible in 
the form $\xi^{\perp}$ for suitable 
$0\neq \xi \in D^{\perp}$. Then we have, by 
\cite[4.C.3]{HLbook},
$$M^{\mu s}_H(X,v) = M^{\mu ss}_H(X,v).$$

The $\xi$ that arise may be characterized in terms of $r$ \cite[4.C.2]{HLbook}:
\begin{exam}
Suppose that $r=2$ and $E$ is strictly $\mu_H$-semistable.

One possibility is extensions
$$0 \rightarrow \cO_X(L) \rightarrow E \rightarrow
\cO_X(D-L) \rightarrow 0, 
$$
where $L$ is a divisor with
$$
H\cdot L = H\cdot (D-L), \quad 
\dim \mathrm{Ext}^1(\cO_X(D-L),\cO_X(L))=2.
$$
Here $\xi = 2L-D \in H^{\perp}$ satisfies $\xi^2=-8$.
Writing $\hat{s}=v(L)$ we have: 
$$
\begin{array}{c|cc}
   & v & \hat{s} \\
\hline
v & 0 & 0 \\
\hat{s} & 0 & -2
\end{array}
$$
Thus $\hat{s}$ gives rise to a $(-2)$-class in the
Picard group of $M_{H'}(X,v)$ for $H'$ a polarization
outside the walls; typically this is the class of a
rational curve isomorphic to $\bP(\mathrm{Ext}^1(\cO_X(D-L),\cO_X(L)))$, contracted in $M_H^{\mu ss}(X,v)$ with
complement $M_H^{\mu s}(X,v)$.  

The other possibility is extensions 
\begin{equation} \label{ext:ideal}
0 \rightarrow \cO_X(L) \rightarrow E \rightarrow
I_x(D-L) \rightarrow 0,
\end{equation}
where $x\in X$ and $L$ is a divisor
$$
H\cdot L = H\cdot (D-L), \quad 
\dim \mathrm{Hom}(L,E)=1.
$$
Here $\xi = 2L-D \in H^{\perp}$ satisfies $\xi^2=-4$, writing $\hat{s}=v(L)$ we have: 
$$
\begin{array}{c|cc}
   & v & \hat{s} \\
\hline
v & 0 & -1 \\
\hat{s} & -1 & -2
\end{array}
$$
Here $M^{\mu s}_H(X,v)=\emptyset$ reflecting 
the fact that the extension (\ref{ext:ideal}) may be
trivial or nontrivial. The resulting coarse moduli space
is isomorphic to $X$; this is called a ``totally semistable''
wall.
\end{exam}

This dichotomy in the wall types is typical and 
explained in \cite[\S 12]{Bridge} (for two-dimensional
moduli spaces) and \cite[Th.~5.7]{BMacri} (in general);
we are grateful to Bayer for pointing out this framework.
The possible walls are all associated with spherical
classes $\hat{s}$ with $v(\hat{s})^2=-2$ of two types:
\begin{itemize}
\item {\em contracting walls:}
$$
\begin{array}{c|cc}
   & v & \hat{s} \\
\hline
v & 0 & 0 \\
\hat{s} & 0 & -2
\end{array}
$$
where $M_H^{\mu ss}(X,v)$ has a contractible
$(-2)$-class in its Picard group;
\item {\em totally semistable walls:}
$$
\begin{array}{c|cc}
   & v & \hat{s} \\
\hline
v & 0 & -r \\
\hat{s} & -r & -2
\end{array}, \quad r\ge 1,
$$
where $M_H^{\mu s}(X,v)$ is empty but the coarse moduli space is left unchanged.
\end{itemize}
Bridgeland \cite{Bridge} elucidates the typical behavior;
we refer the reader to \cite[\S 6]{BMacri} for details
of the derived equivalences associated with
wall crossing arising as compositions of spherical 
twists associated the $(-2)$-classes.

\begin{exam}
Mukai \cite[3.8]{MukaiTata} offers examples of the
second type. Let $F$ be a rigid vector bundle with
$v(F)=\hat{s}$, $r$ its rank, and $E$ the kernel 
of evaluation at a skyscraper sheaf at $x\in X$:
$$ 0 \rightarrow E \rightarrow  F^{\oplus r} \rightarrow
\bC(x) \rightarrow 0.$$
Note that $E$ has local cohomology at $x$ and thus
cannot be locally-free.  
\end{exam}

Suppose that the polarization varies over the ample cone. 
As we cross walls of either type, the moduli spaces associated with adjacent chambers are naturally isomorphic. 
Even for a contracting wall, the minimal resolution of
the nodal moduli space is naturally isomorphic to 
moduli spaces associated with each side.
Since the ample cone is simply-connected,
for all ample $H_1$ and $H_2$ we obtain natural isomorphisms
\begin{equation} \label{getbeta}
\beta_{H_2,H_1}: M_{H_1}^{\mu ss}(X,v) 
\stackrel{\sim}{\rightarrow} M_{H_2}^{\mu ss}(X,v);
\end{equation}
see the discussion following \cite[Th.~1.1]{BMacri}.
This is an instance of the general phenomenon that wall-crossing induces birational maps among moduli spaces of vector bundles on surfaces \cite[4.C.7]{HLbook}. 
However, the universal sheaves over these moduli spaces
-- and the derived equivalences they induce -- do 
vary from chamber to chamber (see\cite[Th.~1.1(b)]{BMacri}).
In particular,  explicit formulas as in Proposition~\ref{prop:yoshioka} are not available for higher rank K3 surfaces.

\

An application of wall-crossing, and a template
for our results in Section~\ref{sect:skew}, is the following
result of Huybrechts \cite[Prop.~4.1]{HuyJAG}: 
Let $X_1$ and $X_2$ be derived equivalent K3 surfaces. 
Then there exists a moduli space of $\mu_H$-stable locally-free sheaves with universal family
$$\cE \rightarrow M^{\mu s}_H(X_2,v)\times X_2$$
and an isomorphism $X_1 \simeq M^{\mu s}_H(X_2,v)$.

\section{Orientation reversing conjugation}
\label{sect:skew}

We continue to assume that $i$ is an anti-symplectic involution
on a K3 surface $X$. As we have seen,
$$
T(X) \subset \rH^2(X,\bZ)^{i=-1},
$$
with complement $\Pic(X)^{i=-1}$, which is negative definite 
by the Hodge index theorem.  

Recall that
Orlov's Theorem \cite[\S 3]{Orlov} asserts that for K3 surfaces (without group
action) isomorphisms of transcendental cohomology lift to
derived equivalences.
Given K3 surfaces $(X_1,i_1)$ and $(X_2,i_2)$ with
anti-symplectic involutions of the same
type in the sense of Alexeev-Nikulin, the existence of an
isomorphism 
$$T(X_1) \stackrel{\sim}{\ra} T(X_2)$$ 
seldom induces an equivariant derived equivalence; a notable exception 
is the case where the anti-invariant Picard group has rank zero or one. 
We only have that
$$\Pic(X_1)^{i_1=-1}, \quad \Pic(X_2)^{i_2=-1}$$
are stably equivalent -- compatibly with the isomorphism on the
discriminant groups of the transcendental lattices -- but not
necessarily isomorphic.

In light of this, 
we propose an orientation reversing conjugation of actions,
with a view toward realizing isomorphisms of transcendental 
cohomology.  

Assume that 
$\Pic(X_1)^{i_1=-1}$ and $\Pic(X_2)^{i_2=-1}$
are not isomorphic, so
there is no $C_2$-equivariant derived equivalence 
$$
D^b(X_1) \stackrel{\sim}{\ra} D^b(X_2)
$$
taking $i_1$ to $i_2$, by Corollary \ref{coro:antisyminv}.
However, let
$$
\dual_j: D^b(X_j) \stackrel{\sim}{\ra} D^b(X_j), \quad j=1,2,
$$
denote the involution  
$$
\cE_* \mapsto \cE^{\vee}_*.
$$
Note that shift and duality commute with each other and with
any automorphism of the K3 surface.  
The action of $\dual_j$ on the Mukai lattice $\widetilde{\rH}(X_j,\bZ)$
is trivial in degrees $0$ and $4$ and multiplication by $-1$ in degree two. Recall that shift acts via $-1$ in all degrees, so composition with
$\dual_j$ is trivial in degree $2$ and multiplication by $-1$ in 
degrees $0$ and $4$.

We propose a general definition and then
explain how it is related to our analysis
of quadratic forms with involution:

\begin{defi}
\label{defi:dualize}
Let $(X_1,i_1)$ and $(X_2,i_2)$ be smooth
projective varieties with involution,
of dimension $n$ with trivial canonical class.  
They are {\em skew equivalent} if there
is a kernel $\cK$ on $X_1\times X_2$,
inducing an equivalence between $X_1$ and $X_2$,
and a quasi-isomorphism
\begin{equation} \label{eq:functional}
(i_1^*,i_2^*)\cK \stackrel{\sim}{\rightarrow} \cK^{\vee}[n].
\end{equation}
\end{defi}
Note that dualization coincides with 
the relative dualizing complex for both
projections $\pi_1$ and $\pi_2$.
The quasi-isomorphism (\ref{eq:functional}) is involutive 
$$ \cK \mapsto \cK^{\vee}[n] \mapsto (\cK^{\vee}[n])^{\vee}[n]\simeq \cK,$$
i.e., $(i_1,i_2)$ takes $\cK$ to the kernel inducing
the inverse of $\Phi_{\cK}$. Since $\cK$ is simple and the base field
is algebraically closed, the quasi-isomorphism may be normalized so this composition is the identity.

Our first property follows straight from the definition:
\begin{prop} \label{prop:tensorskew}
Suppose that $(X_1,i_1)$ and $(X_2,i_2)$ are as specified
in Definition~\ref{defi:dualize} and $\cK$ induces a skew
equivalence between them. Consider line bundles $L_1$ and $L_2$ on $X_1$ and $X_2$ that are anti-invariant under $i_1$ and $i_2$
$$i_j^*L_j \simeq L_j^{\vee}.$$
Then $\cK \otimes (L_1 \boxtimes L_2)$ also induces a skew equivalence.
\end{prop}

Our next property makes explicit the behavior under duality:
\begin{prop} \label{prop:equiv}
Let $(X_1,i_1)$ and $(X_2,i_2)$ be K3 surfaces equipped with involutions.
Suppose that $\cK$ is a kernel inducing an equivalence between
$X_1$ and $X_2$, with induced Fourier-Mukai transforms
$$\Phi_{\cK}:D^b(X_1) \rightarrow D^b(X_2), \quad
\Psi_{\cK}:D^b(X_1) \rightarrow D^b(X_2).$$
Then the following are equivalent:
\begin{itemize}
\item{
$\Phi_{\cK} \dual_1 i_1^* = \dual_2 i_2^* \Phi_{\cK}$;}
\item{
$i_1^* = \Psi_{\cK} i_2^* \Phi_{\cK}$;}
\item{$\cK$ induces a skew equivalence between $(X_1,i_1)$
and $(X_2,i_2)$.}
\end{itemize}
\end{prop}
\begin{proof}
Recall the interpretations of duality and inverses
of Fourier-Mukai equivalences in Section~\ref{sect:cfm}.  
Let $T_j$ denote the shift on $X_j$.
Applying duality gives to the first expression gives
$$T^{-2}_2 \dual_2 \Phi_{\cK^{\vee}} i_1^* = \dual_2 i_2^* \Phi_{\cK}$$
whence
\begin{equation} \label{eqn:keyform}
T_2^2 \Phi_{\cK^{\vee}} i_1^* = i_2^* \Phi_{\cK}.
\end{equation}
This is equivalent to
$$
\Psi_{\cK} \Phi_{\cK^{\vee}} i_1^*= \Psi_{\cK} T_2^{-2} i_2^* \Phi_{\cK}
$$
and 
$$
T_2^{-2}  i_1^* = \Psi_{\cK} T_2^{-2} i_2^* \Phi_{\cK}
$$
which is the same as
$$i_1^* = \Psi_{\cK} i_2^* \Phi_{\cK}.$$
Now formula (\ref{eqn:keyform}) is equivalent to
$$T_2^2 \Phi_{\cK^{\vee}}= i_2^* \Phi_{\cK} i_1^*$$
i.e., applying $i_1^* \times i_2^*$ transforms 
$\cK$ to $\cK^{\vee}[2]$.  
\end{proof}
The second item in Proposition~\ref{prop:equiv} immediately
yields:
\begin{coro}
Skew equivalence is an equivalence relation on K3 surfaces with
involution.
\end{coro}

Suppose again that $X_1$ and $X_2$ are K3 surfaces and $\cK=\cE[1]$ for a universal vector bundle 
$$\cE \rightarrow X_1 \times X_2$$
associated with an isomorphism $X_1=M_v(X_2)$.
Then relation (\ref{eq:functional}) (with $n=2$) translates
into 
\begin{equation} \label{eq:funcVB}
i_2^*\cE_{i_1(x_1)}\simeq (\cE_{x_1})^{\vee}.
\end{equation}

\begin{theo}
\label{theo:skew}
Let $(X_1,i_1)$ and $(X_2,i_2)$ be K3 surfaces 
with involutions. Then the following are equivalent
\begin{itemize}
\item{$(X_1,i_1)$ and $(X_2,i_2)$ are skew derived equivalent;}
\item{there exists an orientation-preserving 
equivalence of Mukai lattices
$$
\phi: \widetilde{\rH}(X_1,\bZ) \longrightarrow \widetilde{\rH}(X_2,\bZ),
$$
satisfying
\begin{equation} \label{eq:CFMTfunc}
\phi(i_1^*(v^{\vee}))=(i_2^*\phi(v))^{\vee}.
\end{equation}
}
\end{itemize}
\end{theo}
As duality and pullback 
commute with each other, the order of these
operations in (\ref{eq:CFMTfunc}) is immaterial.
Furthermore, if $\phi$ satisfies this relation
then so does $-\phi$.  

\begin{rema} \label{rema:twist}
We are not asserting that each cohomological 
equivalence satisfying (\ref{eq:CFMTfunc}) arises from a 
skew equivalence. 
Suppose that $X_1=X_2=X$ with the same involution $i$.  Consider
the spherical twist associated with $\cO_X$ with kernel 
$\cI_{\Delta}[1]$ and cohomology matrix
\begin{equation}
\tau_{\cO_X}:=\left( \begin{matrix} 0 & 0 & -1 \\
				  0 & \mathrm I & 0 \\
				  -1 & 0 & 0  \end{matrix} \right).
\end{equation}
Neither $\tau_{\cO_X}$ nor $-\tau_{\cO_X}$ is obviously realized by a kernel with the requisite self-duality property. Of course, the identity induces a skew equivalence of $(X,i)$ with itself!
Suppose now that $(X_1,i_1)$ and $(X_2,i_2)$ are
arbitrary K3 surfaces with involution. Given $\phi$ satisfying
(\ref{eq:CFMTfunc}), we may pre-compose or post-compose
with $\tau_{\cO_{X_1}}$ or $\tau_{\cO_{X_2}}$ to get another
matrix with the same property.   
\end{rema}

\begin{proof}[Proof of Theorem~\ref{theo:skew}]
The forward implication is clear. Indeed, the 
cohomological Fourier-Mukai transform $\phi_{\cK}$ of a skew equivalence satisfies
$$(i_1,i_2)^* \phi_{\cK} = \phi_{\cK^{\vee}}$$
but $\phi_{\cK^{\vee}}$ differs from $\phi_{\cK}$
by the involution acting via $+1$ on $\rH^0$ and
$\rH^4$ and $-1$ on $\rH^2$.  Thus 
$$\phi_{\cK}: \widetilde{\rH}(X_1,\bZ) \longrightarrow
\widetilde{\rH}(X_2,\bZ)$$
satisfies relation (\ref{eq:CFMTfunc}).

For the reverse implication, we consider the
cohomological Fourier-Mukai transform
$$\phi:
\widetilde{\rH}(X_1,\bZ) \longrightarrow
\widetilde{\rH}(X_2,\bZ).
$$
Set 
$$v_0 := \phi(0,0,1)=(r,a\ell,s),
$$
where $\ell \in \Pic(X_2)$ is primitive and $a\in \bN$. 
\begin{itemize}
\item{
The relation (\ref{eq:CFMTfunc}) implies that $i_2^*\ell=-\ell$,
which means that $\ell^2<0$ if $\ell \neq 0$.
(The Hodge index theorem implies that the intersection form on 
the anti-invariant divisors is negative definite.)}
\item{
Writing $\phi(1,0,0)=(r',D',s')$ we have
$$a (\ell\cdot D') - rs'-sr' = -1$$
whence $\gcd(r,s,a\ell \cdot D')=1$ for some anti-invariant
divisor $D'$ on $X_2$.
Hence $\gcd(r,s,a)=1$ as well.}
\item{If $\ell \neq 0$ then both $r$ and $s$ are nonzero as $v_0$ is
isotropic. 
If $\ell=0$ then $r=0$ or $s=0$ but both cannot vanish.
After applying a twist $\tau_{\cO_{X_2}}$ we may assume that 
$r\neq 0$.}
\item{
If $r<0$, we may replace $\phi$ by $-\phi$. From now on,
we therefore assume $r>0$.}
\end{itemize}

We follow \S 4 of \cite{HuyJAG} to reduce to circumstances where the wall-crossing analysis
of Section~\ref{sect:crossing} may be carried out.

\ 

\noindent {\bf Case I:} $\Pic(X_2)^{i_2=-1}=0$ \\
This case -- with $\ell=0$ -- was addressed above.

\ 

\noindent {\bf Case II:} $\Pic(X_2)^{i_2=-1}=1$\\
Taking $\ell$ to be the generator, all the possible
equivalences are realized with Mukai vectors
$$v_0=(r,\ell,s), \quad \gcd(r,s)=1.$$
Indeed, this follows from Corollary~\ref{coro:counting}:
Writing $\ell^2=-2d$ and factoring $d=\prod_{j=1}^m p_j^{e_j}$
into distinct primes, we see that the automorphism group of $d(\bZ \ell)$ is $C_2^m$. 

Thus it suffices to consider
Mukai vectors with primitive first Chern class,
where wall crossing applies.

\ 

\noindent {\bf Case III:} $\Pic(X_2)^{i_2=-1}\ge 2$ \\
\begin{enumerate}
\item{Suppose that $v_0=(r,a\ell,s)$, with $\gcd(r,a)=1$. 
Then there exists a anti-invariant divisor $E$ on $X_2$
such that $D=rE+a\ell$ is primitive. In particular,
after tensoring by $\cO_X(E)$ the first Chern class
is primitive. However, tensoring by line bundles
has no impact on $\mu$-stability.}
\item{If only $\gcd(s,a)=1$, then after applying the twist
$\tau_{\cO_{X_2}}$ we have $\gcd(r,a)=1$. }
\item{Suppose that $\gcd(r,a)=\alpha>1$ and write 
$$v_0=(r,a\ell,s)=(\alpha r', \alpha a' \ell, s).$$
As before, choose 
an anti-invariant divisor $E$ such that $a'\ell+r'E$ is 
primitive. Tensoring by $E$ gives
\begin{align*}
\exp(E)v_0 &=
(r, a\ell+rE, \tilde{s}:=s+ a E\cdot \ell + r \frac{E\cdot E}{2})\\
&=(r,\alpha(a'\ell+r'E),\tilde{s}).
\end{align*}
Now
$$1=\gcd(r,a,s)=\gcd(r,a,\tilde{s})=\gcd(\alpha,\tilde{s})$$
so we are reduced to the previous case.
}
\end{enumerate}
To summarize, up to twists by $\cO_{X_2}$ that have
no impact on our final result, for each Mukai vector
$v$ inducing a derived equivalence we may always achieve 
$$M_H^{\mu s}(X_2,v)=M_H^{\mu ss}(X_2,v)$$ 
for polarizations $H$ avoiding walls.
This completes Case III.

\

We return to the situation where $v_0=(r,a\ell,s)$ with $\ell$ anti-invariant, applying the wall-crossing technique of Section~\ref{sect:crossing}.
Consider $M^{\mu s}_H(X_2,v_0)$, a K3 surface derived-equivalent to $X_2$, where $H$ is a polarization on $X_2$ avoiding the walls. We produce an involution $j$
on this moduli space by composing isomorphisms
$$M^{\mu s}_H(X_2,v_0) \rightarrow
M^{\mu s}_H(X_2,-v_0) \rightarrow
M^{\mu s}_{i_2^*H}(X_2,v_0) \rightarrow
M^{\mu s}_{H}(X_2,v_0),$$
where the first isomorphism is induced by
duality, the second is induced by $i_2$, and the
third is $\beta_{i_2^*H,H}$ introduced in (\ref{getbeta}).
To see that this is an involution, observe that
$\beta_{H,i_2^*H}=\beta_{i_2^*H,H}^{-1}$
and use the fact that $i_2$ and duality are involutive
and commute with each other.

We analyze how $j$ acts on the cohomology
\begin{equation} \label{gettheta}
\rH^2(M^{\mu s}_H(X_2,v_0),\bZ) 
= (v^{\vee}_0)^{\perp}/ \bZ v^{\vee}_0.
\end{equation}
The isomorphism $\beta$ allows us fix these identifications
as $H$ varies, even as we cross walls. 
The action of $i_2$ on the Mukai lattice and
the associated cohomology groups is given by 
functoriality; recall that $i_2$ takes $v_0$ to 
its dual. For dualization
$$
M^{\mu s}_H(X_2,v_0) \rightarrow M^{\mu s}_H(X_2,-v_0)
$$ the
action is
\begin{align*}
(v^{\vee}_0)^{\perp}/ \bZ v^{\vee}_0 & \rightarrow 
(v_0)^{\perp}/ \bZ v_0 \\
   \gamma & \mapsto -\gamma^{\vee}.
\end{align*}
Indeed, the construction of (\ref{gettheta}) in 
\cite[8.1.1]{HL} and Serre duality for K3 surfaces -- modulo shift by two, the cohomology of the dual of a sheaf is the dual of its cohomology -- shows this is the induced mapping.   
To conclude $j^*$ acts as follows:
\begin{itemize}
\item{
multiplication by $+1$ on 
$$\rH^2(X_2,\bZ)^{i_2=1} \subset \rH^2((M^{\mu s}_H(X_2,v_0)),\bZ);$$}
\item{
multiplication by $-1$ on 
$$(v_0^{\vee})^{\perp} \cap \rH^0(X_2,\bZ) \oplus \rH^4(X_2,\bZ);
$$}
\item{
multiplication by $-1$ on 
$$(v_0^{\vee})^{\perp} \cap \rH^2(X_2,\bZ)^{i_2=-1}.$$}
\end{itemize}

We follow \cite[p.~235]{HuFMTAG}. 
Looking at the composed cohomological Fourier-Mukai transforms
$$\widetilde{\rH}(X_1,\bZ) \stackrel{\sim}{\ra} \widetilde{\rH}(X_2,\bZ) 
\stackrel{\sim}{\ra} \widetilde{\rH}(M^{\mu s}_H(X_2,v_0),\bZ),$$
which takes $(0,0,1)$ to $(0,0,1)$, the Torelli theorem guarantees
that $X_1\simeq M^{\mu s}_H(X_2,v_0)$. The function relation (\ref{eq:CFMTfunc})for $\phi$ guarantees that $i_1$ coincides with $j$ under this isomorphism.

Our moduli space admits a universal sheaf \cite[Prop.~10.20]{HuFMTAG}
$$\cE \rightarrow X_2 \times M^{\mu s}_H(X_2,v_0),$$
unique up to tensor product by line bundles on the moduli space.
On the other hand, 
$$(i_2,I)^*\cE^{\vee} \rightarrow X_2 \times M^{\mu s}_{i_2^*H}(X_2,v_0),$$
is also a universal sheaf, as is 
$$(i_2,j)^*\cE^{\vee} \rightarrow X_2 \times M^{\mu s}_H(X_2,v_0).$$
Applying the isomorphism with $X_1$, we obtain
$$
(i_2,i_1)^*\cE^{\vee} \simeq \cE \otimes L_1,
$$
for some line bundle $L_1$ on $X_1$.  
This is equivalent to
$$(i_2,i_1)^*\cE \simeq \cE^{\vee} \otimes L_1^{\vee}$$
and
$$\cE \simeq (i_2,i_1)^* \cE^{\vee} \otimes i_1^*L_1^{\vee}$$
whence $L_1$ is necessarily symmetric under $i_1$. 

Rescaling $\cE \mapsto \cE \otimes N_1$, for $N_1$ a
line bundle on $X_1$, takes
$$L_1 \mapsto L_1 \otimes N_1 \otimes i_1^*N_1.$$
{\em A priori}, the obstruction to obtaining the
relation 
$$(i_2,i_1)^*\cE^{\vee} \simeq \cE$$
is a cocycle in $\rH^2(\left<i_1\right>,\Pic(X_1))$. 
However, any such obstruction would be visible on cohomology
and thus is precluded by the relation (\ref{eq:CFMTfunc}).
\end{proof}

\begin{coro} \label{coro:altskew}
Under the assumptions above, the functors $\dual_1 \circ i_1$
and $\dual_2 \circ i_2$ are $C_2$-equivariantly derived equivalent.  
\end{coro}
This motivates the formulation of Proposition~\ref{prop:bootstrap}.

\begin{rema}
As we recalled in Section~\ref{sect:genK3},
derived equivalences respect orientations
on the Mukai lattice \cite{HMS}. The orientation reversing  
conjugation violates the orientation condition,
in a prescribed way. Duality is the archetypal orientation-reversing 
Hodge isogeny.  
\end{rema}

In Sections~\ref{sect:ratq} and \ref{sect:genEnr} we give examples of such equivalences.

\section{Rational quotients and skew equivalence}
\label{sect:ratq}

Our first task is to give examples of skew equivalences
using Theorem~\ref{theo:skew}.
We remind the reader to consult Proposition~\ref{prop:bootstrap} for the relevant lattice machinery.

The simplest examples are in rank two.
Take $(X_1,h_1)$ and $(X_2,h_2)$ to be degree-two K3 surfaces, with associated involutions
$i_1$ and $i_2$, such that $T(X_1)\simeq T(X_2)$. 
Suppose that
$$\Pic(X_j)^{i_j}=\bZ \ell_j,\quad  \ell_j^2=-d;$$
note that $\Pic(X_j)$ is either $\left<h_j,\ell_j\right>$ or $\left<h_j,\frac{h_j+\ell_j}{2}\right>$, i.e., the
distinguished $2$-elementary subgroup is trivial
or cyclic. By Corollary~\ref{coro:counting}, possible examples correspond to isomorphisms 
$$d(\bZ \ell_1) \simeq d(\bZ \ell_2)$$
preserving the distinguished subgroup -- a 
vacuous condition as the discriminant group is cyclic.
Thus $(X_1,i_1)$ and $(X_2,i_2)$ are skew equivalent. 
\begin{rema}
In many examples, $M^{\mu s}_{h_j}(X_j,v_0)$ is automatically compact for $v_0=(r,\ell,s)$, with $r<|s|$
and $\gcd(r,s)=1$,
because $h_j$ happens not to lie on a wall.
\end{rema}

The next group of examples arise from nontrivial
stable isomorphisms. We exhibit 
lattice-polarized K3 surfaces with involution $(X_1,i_1)$
and $(X_2,i_2)$, such that the anti-invariant Picard groups
are stably equivalent but inequivalent. 

Specifically, we assume $X_1$ and $X_2$ are degree two 
K3 surfaces with
$$
\Pic(X_j) = \bZ h_j \oplus A_j(-1), \quad h_j^2=2,
$$
where the involutions fix the $h_j$ and reverse
signs on $A_j$'s.    
If $A_1$ and $A_2$ are stably-equivalent, inequivalent positive
definite lattices then $(X_1,i_1)$ and 
$(X_2,i_2)$ are skew equivalent.  

In contrast to ordinary equivalences 
(see \ref{coro:antisyminv}) there are anti-symplectic 
actions with nontrivial {\em skew} equivalences.
The resulting quotients are rational surfaces, indeed,
$\bP^2$.  

\begin{exam}[Explicit matrices]
The matrices, in the basis $p_j,q_j$, for $j=1,2$, are given by
$$
A_1:=
\begin{pmatrix} 4 & 1 \\ 1 & 12
\end{pmatrix},
\quad 
A_2:=
\begin{pmatrix} 		6 & 	1 \\
1	& 8 
\end{pmatrix}.
$$
We extract a stable isomorphism
$$
A_1 \oplus \rU \simeq A_2 \oplus \rU,
\quad 
\rU=\left<u_1,v_1\right>,  \text{ with matrix }
\begin{pmatrix} 0 & 1 \\ 1 & 0 \end{pmatrix}.
$$

First, we give an isomorphism
$$
A_1 \oplus \left<e_1\right> \simeq A_2 \oplus \left<e_2\right>, 
\quad e_1^2 = -2. $$
We put
$$
p_1 \mapsto p_2 + e_2, 
$$
and claim that the orthogonal complements to these are equivalent indefinite lattices. Indeed, 
$$
p_1^{\perp} = \left<p_1-4q_1, e_1\right>= 
\begin{pmatrix} 188 & 0 \\ 0 & -2  \end{pmatrix},
$$
\begin{align*}
(p_2+e_2)^{\perp}  & = \left<p_2-6q_2, 2q_2 + e_2\right> 
= \begin{pmatrix}   282  & -94 \\-94 & 30  \end{pmatrix} \\
& = \left<p_2+3e_2, 2q_2+e_2\right> = \begin{pmatrix}  -12 & -4 \\ -4 & 30 
\end{pmatrix}
\end{align*}
These are equivalent via Gaussian cycles of reduced forms
$$
\begin{array}{ccccccccc}
	 & 0	&   &  	18  &       & 8	&   & 4 & \\
188	 &      &-2	&       &	26	&   &-12&   &30
\end{array}
$$
where the indicated basis elements are
$$
p_1-4q_1, \quad e_1, \quad  p_1 - 4q_1 - 9e_1,	\quad p_1-4q_1-10e_1, \quad
2(p_1-4q_1) - 19e_1.
$$
The composed isomorphism is
\begin{align*}
p_1-4q_1-10e_1    & \mapsto p_2+3e_2, \\ 
2(p_1-4q_1)-19e_1 &   \mapsto 2q_2+e_2 \\
p_1 & \mapsto 		p_2 + e_2\\
e_1 & \mapsto 		(2q_2+e_2) - 2(p_2+3e_2) = 2(q_2-p_2) - 5e_2\\
q_1 &   \mapsto 		5(p_2-q_2) + 12 e_2.
\end{align*}
\end{exam}

We extend the isomorphism above where $e_i = u_i-v_i$
\begin{align*}
u_1+v_1 &\mapsto	u_2+v_2\\
u_1-v_1 &\mapsto	2(q_2-p_2) - 5(u_2-v_2)\\
p_1 	&\mapsto 	p_2 + (u_2-v_2)\\
q_1		& \mapsto 	5(p_2-q_2) + 12(u_2-v_2)
\end{align*}
whence we have
\begin{align*}
u_1 	&\mapsto 	(q_2-p_2) - 2u_2 + 3v_2\\
v_1	 &\mapsto	(p_2-q_2) + 3u_2 - 2v_2.
\end{align*}

\section{Nikulin involutions}
\label{sect:nik}

\subsection*{General properties}
An involution $\iota$ on a K3 surface $X$ over $\bC$ preserving the symplectic form is called a {\em Nikulin} involution. We recall basic facts concerning such involutions, following \cite{Geemen-S}:
\begin{itemize}
\item $\iota$ has 8 isolated fixed points;
\item the (resolution of singularities) $Y\rightarrow X/\iota$ is a K3 surface fitting into a diagram
$$\begin{array}{ccl}
X & \stackrel{\beta}{\leftarrow} & \widetilde{X} \\
\downarrow &   &  \downarrow  \scriptstyle{\pi} \\
X/\iota & \leftarrow & Y 
\end{array}
$$
where $\beta$ blows up the fixed points and the vertical arrows
have degree two;
\item the action of $\iota$ on $\rH^2(X,\bZ)$ is uniquely determined, and there is a decomposition
$$
\rH^2(X,\bZ) = (\rU^{\oplus 3})_1 \oplus (\rE_8(-1) \oplus \rE_8(-1))_P,
$$
where the first term is invariant and the second is a permutation
module for $\iota$;
\item the invariant and the anti-invariant parts of $\rH^2$ take the form:
$$
\rH^2(X,\bZ)^{\iota=1} \simeq \rU^3\oplus \rE_8(-2), \quad 
\rH^2(X,\bZ)^{\iota=-1} = \rE_8(-2)
$$
\end{itemize}
Let $E_1,\ldots,E_8$ denote the exceptional divisors of $\beta$
and $N_1,\ldots,N_8$ the corresponding $(-2)$-curves on $Y$.
The union $\cup N_i$ is the branch locus of $\pi$
so there is a divisor
$$\hat{N}=(N_1+ \cdots + N_8)/2$$
saturating $\left<N_1,\ldots,N_8\right> \subset \Pic(Y)$; the minimal primitive sublattice containing these divisors is called the {\em Nikulin} lattice, and is denoted by $\rN$. 
We have \cite[Prop.~1.8]{Geemen-S}
$$
\begin{array}{rcl}
\pi_*: \rH^2(\widetilde{X},\bZ) & \ra & \rH^2(Y,\bZ) \\
    \rU^3 \oplus \rE_8(-1) \oplus \rE_8(-1) \oplus \left<-1\right>^8 & \ra &   \rU(2)^3 \oplus \rN \oplus \rE_8(-1) \\
                (u,x,y,z) & \mapsto & (u,z,x+y)
\end{array}
$$
and 
$$
\begin{array}{rcl}
\pi^*: \rH^2(Y,\bZ) & \ra & \rH^2(\widetilde{X},\bZ) \\
\rU(2)^3 \oplus \rN \oplus \rE_8(-1) & \ra & 
\rU^3 \oplus \rE_8(-1) \oplus \rE_8(-1) \oplus \left<-1\right>^8 \\
(u,n,x) & \mapsto & (2u,x,x,2\tilde{n})
\end{array}
$$
where if $n=\sum n_iN_i$ then $\tilde{n}=\sum n_i E_i$.
Thus we obtain a distinguished saturated sublattice
$$
\rE_8(-2) \subset \Pic(X)
$$
that coincides with the $\iota=-1$ piece.

\begin{prop} \label{prop:IdentifyNik}
Fix a lattice $\rL$ containing $\rE_8(-2)$ as a primitive sublattice;
assume $\rL$ arises as the Picard lattice of a projective K3 surface.
Then there exists a K3 surface $X$ with Nikulin involution
$\iota$ such that 
$$
\rL = \Pic(X) \supset \Pic(X)^{\iota=-1} = \rE_8(-2).$$
\end{prop}

\begin{proof}
Let $\rA$ denote the orthogonal complement of $\rE_8(-2)$
in $\rL$. There is a unique involution $\iota$ on $\rL$ with
$$\rL^{\iota=1}=\rA, \quad \rL^{\iota=-1}=\rE_8(-2).$$
Now $\iota$ acts trivially on $d(\rL)$ -- keep in mind that $d(\rE_8(-2))$ is
a two-elementary group -- so we may naturally extend $\iota$ to the 
full K3 lattice. (It acts trivially on $\rL^{\perp}.$)
These lattice-polarized K3 surfaces form our family.

Nikulin \cite[\S 4]{Nik-finite} explains how to get involutions
for generic K3 surfaces with lattice polarization $\rL$. 
Choose a surface
$X$ such that $\Pic(X) = \rL$ -- a very general member of the 
family has this property. Clearly $X$ is projective -- it admits
divisors with positive self-intersection. We claim there is an ample divisor $H \in \rA$. Indeed, the ample cone of $X$ is characterized
as the chamber of the cone of positive divisors by the group
generated by reflections associated with indecomposable $(-2)$-classes $E$ of positive degree \cite{LP}. 
Each $(-2)$-class $E$ is perpendicular to a unique ray in
$$\rA \otimes \bR \cap \{ \text{ cone of positive divisors }\}
$$
generated by an element $a_E \in \rA$. Note that 
$\rA$ cannot be contained in $E^{\perp}$ as 
$\rE_8(-2)$ has no $(-2)$-classes.  We conclude that $\rA$ meets
each chamber in the decomposition of the positive cone -- it cannot
be separated from the ample cone by any of the $\rE^{\perp}$.  

Once we have the ample cone, we can extract the automorphism group of $X$ via the Torelli Theorem:
It consists of the Hodge isometries taking the ample cone to 
itself. In particular, any Hodge isometry fixing $H$ is an
automorphism. Thus $\iota$ is an automorphism of $X$.  
\end{proof}

\begin{prop} \label{prop:fiddle}
Let $\rL$ be an even hyperbolic lattice containing $\rE_8(-2)$
as a saturated sublattice. 
Assume that $d(\rL)$ has rank at most $11$.  
Then $\rL$ is unique in its genus and the
homomorphism 
$${\mathrm O}(\rL) \rightarrow {\mathrm O}(q_{\rL})$$
is surjective.
\end{prop}
The condition on the rank of $d(L)$ is satisfied for Picard lattices
of K3 surfaces $X$. We have
$$\Pic(X) \subset \rU^{\oplus 3} \oplus \rE_8(-1)^{\oplus 2}$$
which has rank $22$; $d(\Pic(X))\simeq d(T(X))$ so both groups are generated by $\le 11$ elements. 
\begin{proof} 
We apply Proposition~\ref{prop:nikulin2}. 
For odd
primes $p$, the conditions are easily checked as the rank $r$ of $\rL$ exceeds the rank of the $p$-primary part $d(\rL)$. If $r\ge 12$ then the discriminant group is generated by $\le 10$ elements and we are done. Thus we focus on the $p=2$ case with $r=9,10$, or $11$.  

Let $\rA$ denote the orthogonal complement to $\rE_8(-2)$ in $\rL$.
The overlattice
$$\rL \supset \rA \oplus \rE_8(-2)$$
corresponds to an isotropic subgroup 
$$\rH \subset d(\rA) \oplus d(\rL)$$
with respect to $q_{\rA} \oplus q_{\rL}$. Projection
maps $\rH$ injectively into each summand -- we may interpret
these projections as kernels of the natural maps
$$d(\rA)\ra d(\rL), \quad d(\rE_8(-2)) \ra d(\rL).$$
Thus $\rH$ is a $2$-elementary group, of rank at most three. It follows
that $d(\rL)$ contains at least five copies of $\bZ/2\bZ$.
Remark~\ref{rema:niksmall} shows this validates the hypothesis
of Proposition~\ref{prop:nikulin2}.
\end{proof}

The assumption on the {\em rank} of the discriminant
groups can be replaced by bounds on its {\em order}
\cite[Cor.~22, p.~395]{ConSl} -- at least for purposes
of showing there is one class in each genus.

\subsection*{Rank nine examples}
We focus on examples with Picard rank nine, following 
\cite[Prop.~2.2]{Geemen-S} which lists the possible lattices.
Suppose that $\Pic(X)^{\iota=1}=\bZ f$ with $f^2=2d$, which
is necessarily ample as there are no $(-2)$-classes in
$$
\Pic(X)^{\iota=-1}=\rE_8(-2).
$$
We have the lattice
$$
\Lambda:=\left( 2d \right) \oplus \rE_8(-2),
$$
for all $d$. For even $d$ we have the index-two overlattice
$\widetilde{\Lambda}\supset \Lambda$, generated by
$$
\frac{f+e}{2},
$$
where $f$ is a generator of $\left( 2d \right)$ and $e \in \rE_8(-2)$
is a primitive element with 
$$\left(e,e\right)= \begin{cases} -4 & \text{ if } d=4m+2 \\
                                  -8 & \text{ if } d=4m. 
                    \end{cases}
$$
We are using the fact that the lattice $\rE_8$ has primitive vectors
of lengths $2$ and $4$. Using the shorthand
$$
q(v)=q_{\rE_8(-2)}(v) \pmod{2\bZ},
$$
elements $0\neq v \in e_8(-2):=d(\rE_8(-2))$ are of two types
\begin{itemize}
\item{$120$ elements $v$ with $q(v)=1$ ($A_1+E_7$ type),}
\item{$135$ elements $v$ with $q(v)=0$ ($D_8$ type).}
\end{itemize}
Note that $\widetilde{\Lambda}$ is the 
unique overlattice such that $\rE_8(-2)$ remains saturated.

\begin{prop}
\label{prop:nik-der}
Let $(X_1,f_1)$ and $(X_2,f_2)$ be polarized K3 surfaces of degree $2d$, 
derived equivalent via specialization of the construction in Remark~\ref{rema:FMpartners}.
If $X_1$ admits a Nikulin involution fixing $f_1$ then
\begin{itemize}
\item{$X_2$ admits a Nikulin involution fixing $f_2$;}
\item{there is an
isomorphism
$$\varphi:X_1 \stackrel{\sim}{\ra}X_2.$$
}
\end{itemize}
\end{prop}

\begin{proof}
The derived equivalence induces an isomorphism of lattices with
Hodge structure
$$\rH^2(X_1,\bZ) \supset f_1^{\perp} \simeq f_2^{\perp}\subset \rH^2(X_2,\bZ),$$
which means that $f_2^{\perp}\cap \Pic(X_2)$ contains a sublattice isomorphic to $E_8(-2)$. Thus there exists a Hodge involution 
$$\iota_2^*:\rH^2(X_2,\bZ) \rightarrow \rH^2(X_2,\bZ)$$
with anti-invariant summand equal to this copy of $\rE_8(-2)$. 
The Torelli Theorem -- see \cite[Prop.~2.3]{Geemen-S} -- shows that
$X_2$ admits an involution $\iota_2:X_2 \rightarrow X_2$. 

Isomorphisms of K3 surfaces specialize in families \cite[ch.~I]{MatMum}.
This reduces us to proving the result when the $X_j$ have Picard rank nine,
putting us in the case of Proposition~\ref{prop:fiddle}.
The Counting Formula of \cite[\S 2]{HLOY} -- using the conclusions
of Proposition~\ref{prop:fiddle} -- implies that all Fourier-Mukai
partners of $X_1$ are isomorphic to $X_1$.  
\end{proof}

\begin{rema}
We are {\em not} asserting that $\varphi^*f_2=f_1$!  Suppose that $X_1$ and $X_2$ have Picard rank nine, the minimal possible rank.
Then 
$$\varphi^* f_2 \equiv \alpha f_1 \pmod{\rE_8(-2)}$$
where $\alpha \pmod{4d}$ is the corresponding solution 
to congruence (\ref{squarecond}).
\end{rema}

Thus we obtain nontrivial derived equivalence among Nikulin surfaces even in rank nine!

\subsection*{Rank ten examples}

Turning to rank ten, we offer a generalization of \cite[Prop.~2.3]{Geemen-S}:
\begin{prop} \label{prop:MakeNik}
Fix a rank two indefinite even lattice $\rA$ and an even extension
$$\rL \supset \rA \oplus \rE_8(-2)$$
invariant under $\iota$; here $\iota$ fixes $\rA$ and
acts by multiplication by $-1$ on $\rE_8(-2)$.
Then there exists a K3 surface $X$ with Nikulin involution
$\iota$ such that 
$$
\rA = \Pic(X)^{\iota=1} \subset \Pic(X) = \rL \supset 
\Pic(X)^{\iota=-1} = \rE_8(-2).$$
\end{prop}

\begin{proof}
The lattice $\rL$ embeds uniquely into the K3 lattice by
Proposition~\ref{prop:nikulin3}.
Proposition~\ref{prop:IdentifyNik} gives the desired K3 surface
with involution.
\end{proof}

We observed in Proposition~\ref{prop:fiddle} that the lattices $\rL$ 
are unique in their genus and admit automorphisms
realizing the full group ${\mathrm O}(d(\rL))$.  
Repeating the reasoning for Proposition~\ref{prop:nik-der} we find:
\begin{prop}
\label{prop:nik-derived-ten}
A K3 surface $X$ with involution $\iota_1$,
produced following Proposition~\ref{prop:MakeNik} applied to $\rA_1$, will have a second involution $\iota_2$ associated with $\rA_2$. 
Moreover $(X,\iota_1)$ and $(X,\iota_2)$ are not equivariantly 
derived equivalent.
\end{prop}

We elaborate on the overlattices $\rL$ arising in the
assumptions of Proposition~\ref{prop:MakeNik}. 
What lattices may arise from a given $\rA$?
Each $\rL$ arises from a $2$-elementary
$$\rH \subset d(\rA) \oplus e_8(-2)$$
isotropic with respect to $q_{\rA} \oplus q_{\rE_8(-2)}$. 

We consider the orbits of 
$$\rH\simeq (\bZ/2\bZ)^2 \subset e_8(-2)$$
under automorphisms of the lattice.
These reflect possible quadratic forms on $(\bZ/2\bZ)^2$.  
We enumerate the possibilities, relying on description of
maximal subgroups of the simple group of ${\mathrm O}^+_8(2)$
(automorphisms of $e_8(-2)$) \cite[p.~85]{Atlas} and subgroups of 
$W(\rE_8)$ (a closely related group) associated with reflections
\cite[Table 5]{DPR}. For the reader's reference, we list the
root systems associated with the subgroups in parentheses:
\begin{enumerate}
\item{isotropic subspaces, where $q|\rH$ is trivial -- $1575$
elements ($\rD_4+\rD_4$ type);}
\item{rank one subspaces, where $q|\rH$ has a kernel, e.g.,
$q(x,y)=x^2$ -- $3780=28\times 135$ elements  ($\rA_1+\rA_1+\rD_6$ type);}
\item{``minus lines'' full rank non-split subspaces, e.g., $q(x,y)=x^2+xy+y^2$ -- $1120=28\cdot 120/3$ elements ($\rA_2+\rE_6$ type);}
\item{full rank split subspaces, e.g., $q(x,y)=xy$ --
$4320$ elements.} 
\end{enumerate}
As a check, the Grassmannian $\Gr(2,8)$ has Betti numbers
$$1 \quad 1 \quad 2 \quad 2 \quad 3 \quad 3 \quad 4 
\quad 3 \quad 3 \quad 2 \quad 2 \quad 1 \quad 1$$ 
and thus, by the Weil conjectures, $10795$ points of $\bF_2$.
Note that 
$$10795 = 1575 + 3780 + 1120 + 4320.$$

\subsection*{What about arbitrary rank?}
Let $\rA_1$ and $\rA_2$ be indefinite lattices of rank $r\ge 2$ in the same
genus. Consider overlattices
$$\rL_1 \supset \rA_1 \oplus \rE_8(-2), \quad
\rL_2 \supset \rA_1 \oplus \rE_8(-2)$$
associated with subspaces $\rH\subset e_8(-2)$ in the same orbit, so 
we have
$d(\rL_1)\simeq d(\rL_2)$.  
It follows that $\rL_1 \simeq \rL_2$ provided the $d(\rL_i)$ have
rank at most $11$ (see Proposition~\ref{prop:fiddle}); this
holds for Picard lattices of $K3$ surfaces.
Assuming $\rL_1$ and $\rL_2$ arise as Picard lattices of K3 surfaces,
we obtain results as in Propositions~\ref{prop:nik-der} and
\ref{prop:nik-derived-ten}.

We conclude with one last observation:
\begin{prop}
The existence of a Nikulin structure for one member of a derived equivalence class induces Nikulin structures on all K3 surfaces in the equivalence class.
\end{prop}
Suppose $X_1$ and $X_2$ are derived equivalent and $X_1$ admits
a Nikulin involution. Proposition~\ref{prop:fiddle} implies
$$\Pic(X_1)\simeq \Pic(X_2)$$
and we obtain a copy of $\rE_8(-2) \subset \Pic(X_2)$. 
Proposition~\ref{prop:IdentifyNik} guarantees $X_2$ admits a Nikulin
involution as well.

\section{Geometric application}
\label{sect:geom}

In this section, we present a geometric application of the study of Nikulin involutions, up to derived equivalence.

Let $(X_1,f_1)$ and $(X_2,f_2)$ denote derived equivalent K3 surfaces of degree $12$, admitting Nikulin involutions $\iota_j:X_j \ra X_j$
with $\iota_j^*f_j=f_j$ for $j=1,2$. We assume Picard groups 
$$\Pic(X_j) = \bZ f_j \oplus \rE_8(-2).$$
Note that the derived equivalence induces natural identifications
between the $\rE_8(-2)$ summands of $\Pic(X_1)$ and $\Pic(X_2)$.  
In particular, we obtain bijections between the fixed-point loci
$$X^{\iota_1}_1=X^{\iota_2}_2.$$
Let $Z_j \subset X_j$ denote triples of fixed points compatible with
these bijections. Assuming the $X_j$ are generic, i.e. defined by
quadratic equations in $\bP^7$, the fixed points are not
collinear.  

Projection from the $Z_j$ gives surfaces
$$\Bl_{Z_j}(X_j) \rightarrow Y_j \subset \bP^4$$
where the blowup normalizes the image of the projection. 
These constructions are compatible with the involutions on each side.

We claim that the construction of \cite{HL} gives a Cremona
transform
$$\phi: \bP^4 \stackrel{\sim}{\dashrightarrow} \bP^4$$
such that
\begin{itemize}
\item{the indeterminacy of $\phi$ is $Y_1$;}
\item{the indeterminacy of $\phi^{-1}$ is $Y_2$;}
\item{$\phi$ is compatible with the involutions $\iota_1$
and $\iota_2$ induced in the $\bP^4$'s.}
\end{itemize}
Indeed, the construction induces an isogeny of $\rH^2(X_1,\bZ)$ and
$\rH^2(X_2,\bZ)$ induced by $\phi$, restricting to an isomorphism
of the primitive cohomology 
$$f_1^{\perp} \stackrel{\sim}{\ra} f_2^{\perp}.$$
The construction entails designating projection loci 
$Z'_j \in X_j^{[3]}$ compatible with the associated
$$X_1^{[3]} \stackrel{\sim}{\dashrightarrow} X_2^{[3]},$$
our stipulation that the $Z_j$ consist of suitable fixed points
gives compatible projection loci. 

Suppose that $\phi: \bP^n  \stackrel{\sim}{\dashrightarrow} \bP^n$ is 
birational and equivariant for the action of a finite group $G$. 
In this case, \cite[Thm. 1]{KT-map} introduces a well-defined invariant
\begin{equation}
\label{eqn:cg}
C_G(\phi) := \sum_{\substack{E\in \Ex_G(\phi^{-1})\\ \mathrm{gen. stab}(E)=\{1\}}}
[E\actsfromright G] 
-\sum_{\substack{D\in \Ex_G(\phi)\\ \mathrm{gen. stab}(D)=\{1\}}} 
[D\actsfromright G] \in \bZ[\Bir_{G,n-1}],
\end{equation}
taking values in the free abelian group on $G$-birational isomorphism classes of algebraic varieties of dimension $n-1$. 
In this case, the terms are the projectivized normal bundles of $Y_1$ and $Y_2$, taken with opposite signs.  It is worth mentioning that the underlying K3 surfaces $X_1$ and $X_2$ are isomorphic by Proposition~\ref{prop:nik-der}, and the group actions are conjugate under derived equivalences but not under automorphisms.
The difference of classes of exceptional loci in \eqref{eqn:cg} is {\em nonzero} due to Proposition~\ref{prop:ll} below. This gives an instance
where the refinement of the invariant $c(\phi)$ in \cite{LSZ}, \cite{LSh} using group actions yields new information.

\begin{prop}[cf. Thm. 2, \cite{LS}]
\label{prop:ll}
Let $X_1$ and $X_2$ be smooth projective $G$-varieties that are not uniruled. Then any $G$-equivariant stable birational equivalence 
$$
X_1 \times \bP^r \stackrel{\sim}{\dashrightarrow} X_2 \times \bP^s, 
$$
with trivial $G$-action on the second factors, 
arises from a $G$-equivariant birational equivalence
$$
X_1 \stackrel{\sim}{\dashrightarrow} X_2.
$$
\end{prop}

\begin{proof}
Our assumption -- that $X_1$ and $X_2$ are not uniruled -- means
that 
$$X_1 \times \bP^r \ra X_1, \quad X_2\times \bP^s \ra X_2$$
are maximal rationally-connected (MRC) fibrations. 
Since $X_1\times \bP^r \stackrel{\sim}{\dashrightarrow}X_2 \times \bP^s$,
the functoriality of MRC fibrations \cite[IV.5.5]{K-book} gives a natural birational map
$$X_1 \stackrel{\sim}{\dashrightarrow} X_2.$$
When the varieties admit $G$-actions, the induced birational map is
compatible with these actions. 
\end{proof}

\section{Enriques involutions}
\label{sect:genEnr}

Let $S$ be an Enriques surface over $\bC$. Its universal cover is a K3 surface
$X$ with covering involution $\epsilon:X \rightarrow X$, a fixed-point-free
automorphism of order two, called an {\em Enriques involution}. 

The classification of Enriques surfaces $S$ up to derived
equivalence boils down to the classification of
pairs $(X,\epsilon)$ up to $C_2$-equivariant derived
equivalence \cite[\S 6]{BrMaMZ} (and \cite{BrMaJGP} more generally). Derived equivalent Enriques surfaces are isomorphic \cite[Prop.~6.1]{BrMaMZ}.

A number of authors have classified Enriques involutions on a given K3 surface $X$, modulo its automorphisms $\Aut(X)$: 
\begin{itemize}
\item 
Dolgachev \cite{DolgInv} gave the first examples with finite $\Aut(S)$; Kondo \cite{kondo} offered examples of other types. See the Bibliographic Notes of \cite[Ch.~8]{DolgKon} for more history, including early contributions by Fano.  
\item 
Ohashi showed that there 
finitely many $\Aut(X)$-orbits of such involutions. In the Kummer case, the possible quotients are classified by nontrivial elements of the discriminant group of the N\'eron-Severi group $\NS(X)$.
There are 15 on general Kummer surfaces of product type, 31 in a general Jacobian Kummer surface, but the number is generally unbounded \cite{Ohashi}, \cite{Ohashi-jacobian}. 
\item
 Shimada and Veniani consider {\em singular} (i.e.~rank $20$) K3 surfaces; one of their results is a parametrization of $\Aut(X)$-orbits on the set 
of Enriques involutions; the number of such orbits depends only on the genus of the transcendental lattice $T(X)$  \cite[Thm. 3.19]{ShiV}. 
\end{itemize}

These results are based on lattice theory:
two Enriques involutions on a K3 surface $X$ are conjugate via $\Aut(X)$ if an only if the corresponding Enriques quotients are isomorphic 
\cite[Prop. 2.1]{Ohashi}. 

Let 
$$
\mathrm M:=\mathrm U \oplus \mathrm E_8(-1)
$$
be the unique even unimodular hyperbolic lattice of rank 10; we have
$$
\Pic(S)/\mathrm{torsion}\simeq \mathrm M
$$
and 
$$\Pic(X) \supseteq \mathrm M(2)$$
as a primitive sublattice. This coincides with the invariant sublattice $$
\Pic(X)^{\epsilon=1} \subset \Pic(X)
$$ 
under the involution $\epsilon$.  Let $\mathrm N$ denote the orthogonal complement to $\mathrm M$ in $\rH^2(X,\bZ)$,
which coincides with $\rH^2(X,\bZ)^{\epsilon=-1}$; note that
$T(X) \subset \mathrm N$.  
We have
$$
\mathrm N \simeq \rU \oplus \rU(2) \oplus \rE_8(-2)
$$
which has signature $(2,10)$.  
Thus
$$\Pic(X)^{\epsilon=-1} = T(X)^{\perp} \subset \rN$$
has negative definite intersection form.  
The following result gives a criterion for the existence of Enriques involutions \cite[Thm. 1]{keum}, \cite[Thm. 2.2]{Ohashi}, \cite[Thm. 3.1.1]{ShiV}:

\begin{prop} 
\label{prop:enr-exist}
Let $X$ be a K3 surface. Enriques involutions on $X$ correspond to
the following data:
Primitive embeddings 
$$T(X) \subset \mathrm \rN \subset \rH^2(X,\bZ)$$
such that the orthogonal complement to $T(X)$ in $\rN$ does not
contain $(-2)$-classes.  
\end{prop}

In particular, let $X$ be a K3 surface with an Enriques involution. Then:
\begin{itemize}
\item $\rk \Pic(X)\ge 10$, 
\item if $\rk \Pic(X)=10$ then 
there is a unique such involution,
\item if $\rk \Pic(X)=11$ then $\Pic(X)$ is isomorphic to  \cite[Prop. 3.5]{Ohashi}
\begin{itemize}
\item $\rU(2) \oplus \rE_8 \oplus \langle -2n\rangle$, $n\ge 2$, or
\item $\rU\oplus \rE_8(2)\oplus \langle -4n\rangle$, $n\ge 1$.  
\end{itemize}
\end{itemize}

\begin{prop}
\label{prop:enr-derived}
Let $X$ and $Y$ be derived equivalent K3 surfaces. Assume that $X$ admits an Enriques involution. Then $X$ is isomorphic to $Y$.  In particular, 
the existence of an Enriques involution is a derived invariant. 
\end{prop}
\begin{proof}
In Picard rank $\ge 12$, derived equivalence implies isomorphism. If $X$ and $Y$ and derived equivalent of rank $10$ and $X$ admits an Enriques involution, then $T(X) \simeq T(Y)$ and 
$\Pic(X)$ and $\Pic(Y)$ are {\em stably isomorphic}. 
In Picard ranks 10 and 11, it suffices to show that the lattice $\rM(2)$ is unique in its genus and all automorphisms of the discriminant
group $(d(\rM(2)), q_{\rM(2)}))$ lift to automorphisms
of $\rM(2)$. This is implied by \cite[Thm.~1.14.2]{nik-lattice}.  Indeed, \cite[Lem.~3.1.7]{ShiV} shows that $\Pic(X)$
satisfies these two conditions whenever $X$ admits
an Enriques involution.  
\end{proof}

Corollary~\ref{coro:antisyminv} implies (cf.~\cite[\S 6]{BrMaMZ}): 

\begin{prop}
\label{prop:any}
Any 
$C_2$-equivariant derived autoequivalence 
$$
(X,\epsilon_1) \sim (X,\epsilon_2)
$$ 
arises from an automorphism of $X$.
\end{prop}

We observe a corollary of Proposition~\ref{prop:nikulin2}:
Let $(X_1,\epsilon_1)$ and $(X_2,\epsilon_2)$ denote K3 surfaces with
Enriques involutions. 
They are orientation reversing (i.e. skew) conjugate if
\begin{itemize}
\item{$\tau: T(X_1)\stackrel{\sim}{\ra} T(X_2)$ as lattices, with compatible
Hodge structures;}
\item{$\Pic(X_1)^{\epsilon_1=-1}$ and $\Pic(X_2)^{\epsilon_2=-1}$ have the
same discriminant quadratic form.}
\end{itemize}

\

We explore this in more detail in the case of singular (rank $20$) 
K3 surfaces. The existence of involutions on 
singular K3 surfaces is governed by: 

\begin{prop} \cite{Sert}
Let $X$ be a singular K3 surface with transcendental lattice
$T(X)$ of discriminant $d$. 
There is no Enriques involution on $X$ if and only if $d\equiv 3 \pmod 8$ or 
$$
T(X) = \begin{pmatrix} 2 & 0 \\ 0 & 2 \end{pmatrix}, \quad  
\begin{pmatrix} 2 & 0 \\ 0 & 4\end{pmatrix}, \text{ or } \begin{pmatrix} 2 & 0 \\ 0 & 8\end{pmatrix}.
$$
\end{prop}

The ``most algebraic example'', i.e.~the smallest discriminant admitting an Enriques involution, has
$$T(X) \simeq \left( \begin{matrix} 2 & 1 \\ 1 & 4 \end{matrix} \right)
.$$
In this situation there are two possibilities. We write the maximal
sublattices
$$
\rN \subset \Pic(X)$$
such that the involution $\epsilon$ acts via $-1$.  

We follow the notation \cite[Table~3.1]{ShiV}. We consider lattices
$$
N^{144}_{10,7}(2), \quad N^{242}_{10,7}(2)
$$
where
$$
N^{242}_{10,7}(-1) \simeq 
\left( \begin{matrix} 2 & 1 \\ 1 & 4 \end{matrix} \right) \oplus \rE_8
$$
with $\rE_8$ positive definite and
$$N^{144}_{10,7}(2)(-1) \simeq
\left(\begin{matrix} 
2& 1& 1& 0& 1& 0& 0& 0& 0& 0\\
1& 2& 0& 0& 0& 0& 0& 0& 0& 0\\
1& 0& 2& 1& 0& 0& 0& 0& 0& 0\\
0& 0& 1& 4& 0& 0& 0& 0& 0& 0\\
1& 0& 0& 0& 2& 1& 0& 0& 0& 0\\
0& 0& 0& 0& 1& 2& 1& 0& 0& 0\\
0& 0& 0& 0& 0& 1& 2& 1& 0& 0\\
0& 0& 0& 0& 0& 0& 1& 2& 1& 0\\
0& 0& 0& 0& 0& 0& 0& 1& 2& 1\\
0& 0& 0& 0& 0& 0& 0& 0& 1& 2
\end{matrix} \right).$$
According to {\tt magma}, these two lattices are inequivalent but are in the same spinor genus
thus are stably equivalent.

These involutions are not derived equivalent. Indeed, passing to Mukai lattices adds a hyperbolic summand $\rU$
on which the involution acts trivially. However, in the case
at hand we are
stabilizing the $(-1)$-eigenspace. Thus these involutions are ``skew equivalent'' in the sense of Section~\ref{sect:skew}.  

\section{Postscript on involutions in higher dimensions}

There are many papers addressing the structure of involutions
of higher-dimensional irreducible holomorphic symplectic varieties. 
\begin{itemize}
\item
Symplectic involutions of varieties of $K3^{[n]}$-type and their fixed
loci are classified in \cite{KMO}. 
\item For varieties of Kummer type -- arising
from an abelian surface $A$ -- involutions associated with $\pm 1$ 
on $A$ are analyzed in \cite[Th.~4.4]{HTMMJ} and \cite[Th.~1.3]{KMO}.  
\item
Anti-symplectic involutions on varieties of $K3^{[n]}$-type of degree two
are studied in \cite{FMOS}. 
\item
Higher-dimensional analogs of Enriques involutions are studied
in \cite{OS}.  
\item
Involutions on cubic fourfolds -- both symplectic (see \cite{LZ} and \cite{HTcubic}) and anti-symplectic --
are studied in \cite{Marcubic}. The corresponding actions on lattices are
described explicitly.  
\item
Involutions on O'Grady type examples are considered in \cite{MM}.
\end{itemize}
It is natural to consider whether derived equivalences of
involutions on K3 surfaces $X_1$ and $X_2$ may be understood via 
equivalences of the induced involutions on punctual Hilbert schemes and other
moduli spaces.

\bibliographystyle{alpha}
\bibliography{autoeq}

\end{document}